\documentclass[review]{elsarticle}
\usepackage{amsfonts,amssymb,amsmath,latexsym,indentfirst}
\usepackage{lineno,hyperref}
\modulolinenumbers[5]

\usepackage[top=2cm, bottom=2cm, left=2cm, right=2cm]{geometry}
\usepackage{mathtools}
\usepackage{mathabx}
\mathtoolsset{showonlyrefs}

\journal{}

\begin{document}

\begin{frontmatter}

\title{Blow-up of radial solutions for the intercritical inhomogeneous NLS equation}


\author[mymainaddress,mysecondaryaddress]{Mykael Cardoso}
\ead{mykael@ufpi.edu.br}

\author[mymainaddress]{Luiz Gustavo Farah\corref{mycorrespondingauthor}}
\cortext[mycorrespondingauthor]{Corresponding author}
\ead{farah@mat.ufmg.br}

\address[mymainaddress]{Department of Mathematics, ICEx, Universidade Federal de Minas Gerais, Belo Horizonte - MG, Brazil}
\address[mysecondaryaddress]{Department of Mathematics, CCN, Universidade Federal do Piau\'i, Teresina - PI, Brasil}

\begin{abstract}
We consider the inhomogeneous nonlinear Schr\"odinger (INLS) equation in $\mathbb{R}^N$
\begin{align}\label{inls}
i \partial_t u +\Delta u +|x|^{-b} |u|^{2\sigma}u = 0,
\end{align}
where $N\geq 3$, $0<b<\min\left\{\frac{N}{2},2\right\}$ and $\frac{2-b}{N}<\sigma<\frac{2-b}{N-2}$. The scaling invariant Sobolev space is $\dot{H}^{s_c}$ with $s_c=\frac{N}{2}-\frac{2-b}{2\sigma}$. The restriction on $\sigma$ implies $0<s_c<1$ and the equation is called intercritical (i.e. mass-supercritical and energy-subcritical). Let $u_0\in \dot H^{s_c}\cap \dot H^1$ be a radial initial data and $u(t)$ the corresponding solution to the INLS equation. We first show that if $E[u_0]\leq 0$, then the maximal time of existence of the solution $u(t)$ is finite. Also, for all radially symmetric solution of the INLS equation with finite maximal time of existence $T^{\ast}>0$, then $\limsup_{t\rightarrow T^{\ast}}\|u(t)\|_{\dot H^{s_c}}=+\infty$. Moreover, under an additional assumption and recalling that $\dot{H}^{s_c} \subset L^{\sigma_c}$ with $\sigma_c=\frac{2N\sigma}{2-b}$, we can in fact deduce, for some $\gamma=\gamma(N,\sigma,b)>0$, the following lower bound for the blow-up rate
\begin{align*}
c\|u(t)\|_{\dot H^{s_c}}\geq \|u(t)\|_{L^{\sigma_c}}\geq |\log (T-t)|^{\gamma},\,\,\,\textnormal{as}\,\,\,t\rightarrow T^{\ast}.
\end{align*}
The proof is based on the ideas introduced for the $L^2$ super critical nonlinear Schr\"odinger equation in the work of Merle and Rapha\"el \cite{MR_Bsc} and here we extend their results to the INLS setting.
\end{abstract}


\end{frontmatter}

\newcommand{\Real}{\mathbb R}
\newtheorem{thm}{Theorem}[section]
\newtheorem{prop}[thm]{Proposition}
\newtheorem{lemma}[thm]{Lemma}
\newtheorem{rem}[thm]{Remark}
\newtheorem{definition}[thm]{Definition}
\newtheorem{coro}[thm]{Corollary}
\numberwithin{equation}{section}
\newenvironment{proof}{\paragraph{\bf{Proof:}}}{\hfill$\square$}


\section{Introduction}
In this work we consider the initial value problem (IVP) for the inhomogeneous nonlinear Schr\"odinger (INLS) equation
\begin{equation}
\begin{cases}
i \partial_t u + \Delta u + |x|^{-b} |u|^{2 \sigma}u = 0, \,\,\, x \in \mathbb{R}^N, \,t>0,\\
u(0) = u_0,
\end{cases}
\label{PVI}
\end{equation}
for $N\geq 3$, $0<b<\min\left\{\frac{N}{2},2\right\}$ and $\frac{2-b}{N}<\sigma<\frac{2-b}{N-2}$. Note that the case $b = 0$ is the classical nonlinear Schr\"odinger (NLS) equation. The physical relevance of the INLS model \eqref{PVI}  appears naturally in nonlinear optics, we refer the reader to \citet{GILL} and \citet{LIU} for more details.

Throughout this manuscript, we will work with the standard $L^2$-based Sobolev spaces $H^{s}=H^{s}(\mathbb{R}^N)$ and $\dot{H}^{s}=\dot{H}^{s}(\mathbb{R}^N)$, for $s\in \mathbb{R}$, equipped with the norm $\|f\|_{H^{s}}:=\|(1+|\xi|^2)^{\frac{s}{2}}\widehat{f}\|_{L^2}$ and $\|f\|_{\dot{H}^{s}}:=\||\xi|^s\widehat{f}\|_{L^2}$, respectively. In particular, $\dot{H}^{0}={H}^{0}=L^2$. The $\dot{H}^1$ flow admits the energy conservation law given by
\begin{equation}\label{Energy}
E[u(t)] =\frac12 \int  |\nabla u(x,t)|^2\, dx - \frac{1}{2\sigma+2}\int |x|^{-b}|u(x,t)|^{2\sigma+2}\,dx=E[u_0]. 
\end{equation}
Another fundamental conserved quantity for the $L^2$ flow is the mass conservation
\begin{equation}\label{Mass}
M[u(t)] =\int  |u(x,t)|^2\, dx =M[u_0]. 
\end{equation}
Moreover, the scaling symmetry $u(x,t)\mapsto \lambda^{\frac{2-b}{2\sigma}}u(\lambda x, \lambda^2 t)$ plays an important role in our analysis, since it leaves invariant the norm in the homogeneous Sobolev space $\dot{H}^{s_c}$, where $s_c=\frac{N}{2}-\frac{2-b}{2\sigma}$. If $s_c = 0$ (alternatively $\sigma = \frac{2-b}{N}$) the problem is mass-critical and if $s_c=1$ (alternatively $\sigma =\frac{2-b}{N-2}$) it is energy-critical. Finally the problem is mass-supercritical and energy-subcritical or just intercritical if $0<s_c<1$ (alternatively $\frac{2-b}{N}<\sigma<\frac{2-b}{N-2}$).

The well-posedness for the Cauchy problem \eqref{PVI} has been receiving increasing attention over the past years, see for instance \citet{GENSTU, CARLOS, Dinh4, LeeSeo}. Recently, in a joint work with Guzm\'an \cite{CFG20}, the authors studied this  problem in the space $\dot H^{s_c}\cap \dot H^1$, showing local well posedness for $N\geq 3$, $0<b<\min\left\{\frac{N}{2},2\right\}$ and $\frac{2-b}{N}<\sigma<\frac{2-b}{N-2}$. More precisely, it was proved that for all $u_0\in \dot H^{s_c}\cap\dot H^1 $, there exist $T(\|u_0\|_{\dot H^{s_c}\cap\dot H^1})>0$ and a unique solution $u\in C([0,T);\dot H^{s_c}\cap \dot H^1)$ of \eqref{PVI}. Furthermore, let $T^{\ast}>0$ denotes the maximal time of existence for this solution, if $T^{\ast}<+\infty$, then there exist $c, \theta>0$ such that  
\begin{align}\label{bounded}
\|u(t)\|_{\dot H^{s_c}\cap \dot H^1}\geq\frac{c}{(T^{\ast}-t)^{\theta}}.
\end{align}

The existence of solutions with finite maximal time of existence is already known in $H^1$ for the INLS model \eqref{PVI}. As for the classical NLS equation this is a consequence of the following virial identity satisfied by solutions to \eqref{PVI} with initial data $u_0\in \Sigma:=\{f\in H^1;\,\,|x|f\in L^2 \}$
\begin{align}\label{VI}
\frac{d^2}{dt}\int |x|^2|u(x,t)|^2=8(2\sigma s_c+2)E[u_0]-8\sigma s_c\|\nabla u(t)\|_{L^2}^2.
\end{align}
This was obtained by the second author in \cite{Farah}, following the approach developed by \citet{HRasharp} in their study of the intercritical classical NLS equation (case $b=0$, $N\geq 3$ and $\frac{2}{N}<\sigma<\frac{2}{N-2}$). Later, \citet{dinh2017blowup} extended this result assuming radial initial data $u_0\in H^1$, using the ideas of \citet{ogawa1991blow}.

In the present work we are interested in solutions of \eqref{PVI} with initial data in $\dot H^{s_c}\cap \dot H^1$ and finite maximal time of existence. More precisely, first we establish sufficient conditions for the existence of such solutions. In a second step, we investigate the behavior of its $\dot H^{s_c}$ norm. Note that since $\|u\|_{\dot H^{s_c}\cap\dot H^1}=\|u\|_{\dot H^{s_c}}+\|u\|_{\dot H^1}$ this last goal is not a direct consequence of the inequality \eqref{bounded}.

For the intercritical classical NLS equation, this question was studied by \citet{MR_Bsc}. They showed that radially symmetric initial data with non-positive energy is a sufficient condition to deduce the finite time blow-up of the $\dot H^{s_{0}}$ norm, with $s_{0}=\frac{N}{2}-\frac{1}{\sigma}$. Furthermore, recalling that $\dot{H}^{s_{0}} \subset L^{\sigma_{0}}$ with $\sigma_{0}=N\sigma$, they in fact proved the following lower bound for the blow-up rate
\begin{align}
c\|u(t)\|_{\dot H^{s_{0}}}\geq \|u(t)\|_{L^{\sigma_{0}}}\geq |\log (T^{\ast}-t)|^{C_{N,\sigma}},\,\,\,\mbox{as}\,\,\,t\to T^{\ast}.
\end{align}
Note that this type of result breaks down in the $L^2$ critical case ($\sigma=\frac{2}{N}$ and $s_{0}=0$), since the $L^2$ norm is conserved.

Inspired by the work of \citet{MR_Bsc} our aim in this paper is to extend their results to the INLS setting. In the first theorem, we show the existence of non-positive energy solutions with finite maximal time of existence. In particular, this result generalizes the results obtained by the second author in \cite{Farah} and \citet{dinh2017blowup} to the regularity $\dot H^{s_c}\cap \dot H^1$. More precisely, we prove the following.
\begin{thm}\label{scteo1} Let $N\geq 3$, $0<b<\min\{\frac{N}{2},2\}$ and $\frac{2-b}{N}<\sigma<\frac{2-b}{N-2}$. If $u_0\in \dot H^{s_c} \cap\dot H^1 $ is radially symmetric and $E(u_0)\leq 0,$ then the maximal time of existence $T^{\ast}>0$ of the corresponding solution $u(t)$ to \eqref{PVI} is finite.
\end{thm}

The next result describes the behavior of the $\dot H^{s_c}$ norm for any solution of \eqref{PVI} with initial data in $\dot H^{s_c}\cap \dot H^1$ and finite maximal time of existence.
\begin{thm}\label{scteo2}
	Let $s_c=\frac{N}{2}-\frac{2-b}{2\sigma}$ and $\sigma_c=\frac{2N\sigma}{2-b}$ such that $\dot{H}^{s_c} \subset L^{\sigma_c}$. Assume $N\geq 3$, $0<b<\min\left\{\frac{N}{2},2\right\}$ and $\frac{2-b}{N}<\sigma<\frac{2-b}{N-2}$. Given $u_0\in \dot H^{s_c}\cap \dot H^1$ radially symmetric and assume that the maximal time of existence $T^{\ast}>0$ of the corresponding solution $u$ to \eqref{PVI} is finite. Assume that 
\begin{align}\label{esttl}
\left\|\nabla u(t)\right\|_{L^2}\geq \frac{c}{(T^{\ast}-t)^{\frac{1-s_c}{2}}},
\end{align}
for some constant $c=c(N,\sigma)$ and $t$ close enough to $T^{\ast}$. Then there exists $\gamma=\gamma(N,\sigma,b)>0$ such that
	\begin{align}\label{rate}
	c\|u(t)\|_{\dot H^{s_c}}\geq \|u(t)\|_{L^{\sigma_c}}\geq |\log (T^{\ast}-t)|^{\gamma},\,\,\,\,as\,\,t\to T^{\ast}.
	\end{align}
\end{thm} 

The condition \eqref{esttl} is a natural assumption. It is automatically satisfied for the classical NLS equation from the local Cauchy theory in $\dot{H}^{1}$ obtained by \citet{CaWe89} (see the Introduction in \citet{MR_Bsc}). Due to the lack of a local Cauchy theory in $\dot{H}^{1}$ for the INLS model we include this assumption on the statement of Theorem \ref{scteo2}. Moreover, assuming additionally that the initial data $u_0\in H^1\subset \dot H^{s_c}\cap \dot H^1$, the lower bound \eqref{esttl} has been obtained in a recent work by \citet{AT21}. 

The proofs of Theorems \ref{scteo1} and \ref{scteo2} follow the approach introduced by \citet{MR_Bsc}. The first obstacle here is to control the term $|x|^{-b}$ present in the potential part of the energy (see Lemma \ref{lemaradialGN}-$(ii)$). After that, we carry out a careful study to understand how the introduction of the parameter $b$ will affect the rest of the analysis. 

More precisely, to prove Theorem \ref{scteo1}, we first observe that for all functions in $L^{\sigma_c}$ the $L^2$ norm over balls around the origin $\{ |x|\leq R\}$ grows slower than $R^{s_c}$, as $R\to \infty$ (see Lemma Lemma \ref{lemaradialGN}-$(i)$). On the other hand, for a radial solutions with non-positive energy, there exists a radius $R(t)$ (depending on time $t$) such that the $L^2$ norm of the initial data over the ball $\{ |x|\leq R(t)\}$ is at least $cR(t)^{s_c}$ for some universal constant $c>0$ (see Proposition \ref{prop2}-\eqref{tprop2}). Moreover, for a global radial solution with non-positive energy, in view of a global dispersive estimate (see Proposition \ref{prop1}-\eqref{prop1i}), there exists a sequence of times $\{t_n\}_{n=1}^{\infty}$ growing to infinity such that the $L^2$ norm of the gradient goes to zero. This implies that $R(t_n)\to \infty$, reaching a contradiction. 

The proof of Theorem \ref{scteo2} requires a more refined analysis. Indeed, first consider a renormalization $v^{(t)}(\tau)$ of the original solution such that, assuming \eqref{esttl}, exists for large times and where we can apply Propositions \ref{prop1} and  \ref{prop2} uniformly (see Lemma \ref{vteo2}). In this new setting, the problem is reduced to show an specific lower bound on the $L^{\sigma_c}$ norm of the initial data $v^{(t)}(0)$ (see inequality \eqref{N1}). Then, in the worst case scenario, an uniform lower bound for this norm restricted to suitable annulus on space holds (see inequality \eqref{teste}). Finally, summing over a family of disjoint annuli yields the desired result.

The next result is a direct consequence of Theorem \ref{scteo2}.
\begin{coro}\label{cor13}
Let $s_c=\frac{N}{2}-\frac{2-b}{2\sigma}$. Assume $N\geq 3$, $0<b<\min\left\{\frac{N}{2},2\right\}$ and $\frac{2-b}{N}<\sigma<\frac{2-b}{N-2}$. Given $u_0\in \dot H^{s_c}\cap \dot H^1$ radially symmetric and assume that the maximal time of existence $T^{\ast}>0$ of the corresponding solution $u$ to \eqref{PVI} is finite, then
\begin{equation}\label{Bup}
\limsup_{t\rightarrow T^{\ast}}\|u(t)\|_{\dot H^{s_c}}=+\infty.
\end{equation}
\end{coro}
This is in sharp contrast to the $L^2$-critical case ($s_c = 0$), where the mass is preserved, see \eqref{Mass}, and thus \eqref{Bup} cannot occur. So, we extend the situation observed by \citet{MR_Bsc} for the classic NLS equation to the INLS model, for radially symmetric initial data. Moreover, the assumption \eqref{esttl} is only needed to deduce the lower bound on the blow-up rate \eqref{rate} and is not required in Corollary \ref{cor13}.

In our last result, if we assume in addition that $u_0\in H^1 $, we obtain an upper bound on blow-up rate for radially symmetric initial data.
\begin{thm}\label{thmblowdisper}
	Let $N\geq 3$, $0<b<\min\left\{\frac{N}{2},2\right\}$ and $\frac{2-b}{N}<\sigma<\frac{2-b}{N-2}$. Let $u_0\in H^1 $ with radial symmetry and assume that the maximal time of existence $T^{\ast}>0$ for the corresponding solution $u\in C([0,T^{\ast}):H^1 )$ of \eqref{PVI} is finite. Define $\beta=\frac{2-\sigma}{\sigma(N-1)+b}$, then the following space-time upper bound holds
	\begin{align}\label{blowupestdisp}
	\int_{t}^{T^{\ast}}(T^{\ast}-\tau)\|\nabla u(\tau)\|_{L^2}^2\,d\tau\leq C_{u_0}(T^{\ast}-t)^{\frac{2\beta}{1+\beta}},
	\end{align}
for $t$ close enough to $T^{\ast}$.
\end{thm}

This is an application of a localized virial type inequality satisfied by the solutions of the INLS equation (see Lemma \ref{lemintradial}) and its proof is based on the ideas introduced by Merle, Rapha\"{e}l, and Szeftel \cite{MRS2014} for the classical NLS equation.

As a consequence of \eqref{blowupestdisp} we have the following upper bound
	\begin{align}
	\liminf_{t\uparrow T^{\ast}}(T^{\ast}-t)^{\frac{1}{1+\beta}}\|\nabla u(t)\|_{L^2}<+\infty.
	\end{align}
	In particular, there exists a sequence $\{t_n\}_{n=1}^{+\infty}\subset [0,T^{\ast})$ with $t_n\to T^{\ast}$ such that
	\begin{align}
	\|\nabla u(t_n)\|_{L^2}\leq \frac{C}{(T^{\ast}-t_n)^{\frac{1}{1+\beta}}},\,\,\,\, \mbox{ as }\,\,\,\, n\to+\infty.
	\end{align}
	It should be emphasized that it is an open problem to show the previous upper bound for every sequence $\{t_n\}_{n=1}^{+\infty}$ with $t_n\to T^{\ast}$.

This paper is organized as follows. In Section \ref{sec3}, we established a radial interpolation estimate. A virial type estimate and the proof of Theorem \ref{thmblowdisper} is discussed in Section \ref{sec31}. In Section \ref{sec4}, we obtain the key propositions needed to prove our main results. The last two sections are devoted to the proofs of Theorems \ref{scteo1}-\ref{scteo2} and Corollary  \ref{cor13}.

\section{A radial Gagliardo-Nirenberg inequality}\label{sec3}

We first recall a Gagliardo-Nirenberg type inequality proved in  \cite{CFG20}
\begin{align}\label{GNine}
\int |x|^{-b} |u(x)|^{2\sigma+2}\,dx\leq \frac{\sigma+1}{\|V\|_{L^{\sigma_c}}^{2\sigma}}\|\nabla u\|_{L^2}^2\|u\|_{L^{\sigma_c}}^{2\sigma},
\end{align}
where $V$ is a solution to elliptic equation 
\begin{align}\label{elptcpc1}
\Delta V+|x|^{-b}|V|^{2\sigma}V-|V|^{\sigma_c-2}V=0
\end{align}
 with minimal $L^{\sigma_c}$-norm.
 
The main goal of this section is to improve this last inequality in the radial setting and away from the origin. For this purpose, we use the following scaling invariant Morrey-Campanato type semi-norm
\begin{equation}\label{defrho}
\rho(u,R)=\sup_{R'\geq R}\frac{1}{(R')^{2s_c}}\int_{R'\leq |x|\leq 2R'}|u|^2\,dx.
\end{equation}

In the next result, we present the radial interpolation estimate adapted to the INLS setting.

\begin{lemma}\label{lemaradialGN}
	Let $N\geq 3$, $0<b<\min\left\{\frac{N}{2},2\right\}$, $\frac{2-b}{N}<\sigma<\frac{2-b}{N-2}$, $s_c=\frac{N}{2}-\frac{2-b}{2\sigma}$ and $\sigma_c=\frac{2N\sigma}{2-b}$. The following statements hold.
	\begin{itemize}
		\item[(i)]
		There exists a universal constant $c>0$ such that for all $u\in L^{\sigma_c}$ and $R>0$
		\begin{equation}\label{radial1}
		\frac{1}{R^{2s_c}}\int_{|x|\leq R}|u|^2\,dx\leq c\|u\|_{L^{\sigma_c}}^2.
		\end{equation}
Moreover
		\begin{equation}\label{radial2}
		\lim_{R\to+\infty}\frac{1}{R^{2s_c}}\int_{|x|\leq R}|u|^2\,dx=0.
		\end{equation}
		\item[(ii)] For all $\eta>0$, there exists a constant $C_\eta>0$ such that for all $R>0$ and $u\in \dot{H}^{s_c}\cap\dot{H}^{1} $ with radial symmetry the following inequality holds
		\begin{equation}\label{GNradial}
		\int_{|x|\geq R}|x|^{-b}|u|^{2\sigma+2}\,dx\leq \eta\|\nabla u\|_{L^{2}(|x|\geq R)}^{2}+\frac{C_{\eta}}{R^{2(1-s_ c)}}\left\{[\rho(u,R)]^{\frac{2+\sigma}{2-\sigma}}+[\rho(u,R)]^{\sigma+1}\right\}.
		\end{equation}
	\end{itemize}
\end{lemma}
\begin{proof}
	(i) This item was proved by Merle and Raph\"ael in \citep[Lemma 1]{MR_Bsc}. For the convenience of the reader, we also present the proof here. Since $0<s_c<1$, from H\"older's inequality, we have
	\begin{equation}
	\int_{|x|\leq R}|u|^2\,dx\leq \left(\int_{|x|\leq R} |u|^{\sigma_c}\,dx\right)^{\frac{2}{\sigma_c}}\left(\int_{|x|\leq R}1\,dx\right)^{\frac{2s_c}{N}}.
	\end{equation}
	then \eqref{radial1} follows. 
	
	Let $R>A>1$ and divide the integral in the left hand side of \eqref{radial1} in two parts to get
	\begin{align}
	\frac{1}{R^{2s_c}}\int_{|x|\leq R}|u|^2\,dx&\leq\frac{1}{R^{2s_c}}\int_{|y|\leq A}|u|^2\,dx+\frac{1}{R^{2s_c}}\int_{A\leq |y|\leq R}|u|^2\,dx\\ &\leq\frac{1}{R^{2s_c}}\int_{|y|\leq A}|u|^2\,dx+\int_{A\leq |y|\leq R}|u|^2\,dx.
	\end{align}
	Now, given $\varepsilon>0$, choose $R>A>1$ large enough such that
	\begin{align}
	\displaystyle\int_{|y|\geq A}|u|^2\,dx <\frac{\varepsilon}{2}\,\,\,\mbox{and} \,\,\, \displaystyle\left(\frac{A}{R}\right)^{2s_c}<\frac{\varepsilon}{2c}.
	\end{align}
	Thus, using \eqref{radial1} we deduce
	\begin{equation}
	\frac{1}{R^{2s_c}}\int_{|x|\leq R}|u|^2\,dx\leq c\left(\frac{A}{R}\right)^{2s_c}\|u\|_{L^{\sigma_c}}^2+\int_{|y|\geq A}|u|^2\,dx<\varepsilon,
	\end{equation}
which completes the proof of item (i).\\
	(ii) We first assume that $u$ is a smooth radially symmetric function. For $D>0$, define the following set
	$$\mathcal{C}=\{x\in \Real^N;\,D\leq |x|\leq 2D\}.$$
	Let $x_0\in \mathcal{C}$ such that $|u(x_0)|=\|u\|_{L^{\infty}(\mathcal{C})}.$ If there exists $y_0\in \mathcal{C}$ such that
	\begin{equation}\label{condrad}
	|u(y_0)|\leq\frac{1}{2}|u(x_0)|,
	\end{equation}
then
	\begin{align}\label{lemrad1}
	\|u\|_{L^{\infty}}^2&=\frac{4}{3}\left||u(x_0)|^2-\frac{1}{4}|u(x_0)|^2\right|\leq \frac{4}{3}\left||u(x_0)|^2-|u(y_0)|^2\right|.
	\end{align}
	Since $u$ is a radial function we set $u(x)=v(|x|)=v(r)$, then
	\begin{align}\label{lemrad2}
	\int_{\mathcal{C}}|u(x)|^2\,dx=\int_{D}^{2D}\left(\int_{S^{N-1}}|v(r)|^2d\omega\right)r^{N-1}dr \geq cD^{N-1}\int_{D}^{2D}|v(r)|^2\,dr,
	\end{align}
where $S^{N-1}$ denotes the surface of the unit ball in $\mathbb{R}^N$.

Moreover, since $\partial_{x_i}u=\frac{x_i}{r}\,\partial_r v$, using the same argument we also deduce
	\begin{align}\label{lemrad3}
	\int_{\mathcal{C}}|\nabla u(x)|^2
	\,dx\geq cD^{N-1}\int_{D}^{2D}|v'(r)|^2\,dr.
	\end{align}
	Hence, by \eqref{lemrad1}, \eqref{lemrad2},  \eqref{lemrad3} and Cauchy-Schwarz inequality, we get
	\begin{align}
	\|u\|_{L^{\infty}}^{2}&\leq c\left||v(|x_0|)|^2-|v(|y_0|)|^2\right|=c \left|\int_{|y_0|}^{|x_0|} \partial_r |v(r)|^2\,dr\right|\\
	&
\leq c\int_{D}^{2D}|v(r)||v'(r)|\,dr\leq \frac{c}{D^{N-1}} \|\nabla u\|_{L^2(\mathcal C)}\|u\|_{L^2(\mathcal C)}\nonumber.
	\end{align}
	Thus,
	\begin{align}
	\int_{\mathcal C}|x|^{-b}|u|^{2\sigma+2}\,dx &\leq D^{-b}\|u\|_{L^{\infty}(\mathcal C)}^{2\sigma} \|u\|^2_{L^2(\mathcal C)} \leq \frac{c}{D^{\sigma(N-1)+b}}\|\nabla u\|^\sigma_{L^2(\mathcal C)}\|u\|_{L^2(\mathcal C)}^{2+\sigma}\\
	&\leq \frac{c}{D^{(2-\sigma)(1-s_c)}}\|\nabla u\|^{\sigma}_{L^2(\mathcal C)}\left(\frac{1}{D^{2s_c}}\int_{D\leq |x|\leq 2D} |u|^2\,dx\right)^{\frac{2+\sigma}{2}}\nonumber\\
	&\leq \frac{c}{D^{(2-\sigma)(1-s_c)}}\|\nabla u\|_{L^2(\mathcal C)}^{\sigma}[\rho(u,D)]^{\frac{2+\sigma}{2}}.\label{Young}
	\end{align}
	For all $\eta>0$, from the Young inequality, there exists a constant $C_\eta>0$ such that
	\begin{align}
	\int_{\mathcal C}|x|^{-b}|u|^{2\sigma+2}\,dx\leq\eta\|\nabla u\|_{L^2(\mathcal C)}^{2}+C_\eta\frac{1}{D^{2(1-s_c)}}[\rho(u,D)]^{\frac{2+\sigma}{2-\sigma}}.
	\end{align}
	On the other hand, if \eqref{condrad} does not hold for any $y\in \mathcal C$, then
	\begin{align}
	\frac{1}{2}|u(x_0)|\leq |u(y)|
	\end{align}
	then,
	\begin{align}
	\|u\|_{L^{2}(\mathcal{C})}\geq\|u\|_{L^{\infty}(\mathcal C)}\left(\int_{D}^{2D}\left(\int_{S^{N-1}}d\omega\right)r^{N-1}dr\right)^{\frac{1}{2}}\geq  cD^{\frac{N}{2}}\|u\|_{L^{\infty}(\mathcal C)}.
	\end{align}
	Hence, we deduce the following estimate for the potential energy over $\mathcal C$
	\begin{align}
	\int_{\mathcal C}|x|^{-b}|u|^{2\sigma+2}\,dx\leq c\|u\|_{L^\infty(\mathcal C)}^{2\sigma+2}D^{N-b}\leq \frac{c}{D^{N\sigma+b}}\|u\|_{L^2(\mathcal C)}^{2\sigma+2}\leq\frac{c}{D^{2(1-s_c)}} [\rho(u,D)]^{\sigma+1}.
	\end{align}
	Therefore, in both cases, given $\eta>0$ there exists a constant $C_\eta>0$ such that
	\begin{align}\label{GNann}
	\int_{\mathcal C}|x|^{-b}|u|^{2\sigma+2}\,dx\leq \eta \|\nabla u\|_{L^2(\mathcal C)}^{2}+\frac{C_\eta}{D^{2(1-s_c)}}\left\{[\rho(u,D)]^{\frac{2+\sigma}{2-\sigma}}+[\rho(u,D)]^{\sigma+1}\right\}.
	\end{align}
	
	Now, we extend the above estimate for the set $\{x\in \Real^N; |x|\geq R\}$. Indeed, given $R>0$ and $j\in \mathbb{N}$, we first note that
	\begin{align}\label{monrho}
	\rho(u,2^jR)=\sup _{R'>2^jR}\frac{1}{(R')^{2s_c}}\int_{R'\leq |x|\leq2R'}|u|^2\,dx\leq \sup_{R'>R}\frac{1}{(R')^{2s_c}}\int_{R'\leq |x|\leq 2R'}|u|^2\,dx=\rho(u,R).\nonumber
	\end{align}
	Therefore, since
	\begin{align}
	\int_{|x|\geq R}|x|^{-b}|u|^{2\sigma+2}\,dx=\sum_{j=0}^{\infty}\int_{2^jR\leq |x|\leq 2^{j+1}R}|x|^{-b}|u|^{2\sigma+2}\,dx
	\end{align} 
	we obtain, using \eqref{GNann} with $D=2^jR$, that
	\begin{align}
	\int_{|x|\geq R}&|x|^{-b}|u(x)|^{2\sigma+2}\,dx\\&\leq \eta\sum_{j=0}^{\infty}\|\nabla u\|_{L^2(2^jR\leq |x|\leq 2^{j+1}R)}^{2}+C_{\eta}\left\{[\rho(u,R)]^{\frac{2+\sigma}{2-\sigma}}+[\rho(u,R)]^{\sigma+1}\right\}\sum_{j=0}^{\infty}\frac{1}{(2^jR)^{2(1-s_c)}}\nonumber\\
	&\leq \eta\|\nabla u\|_{L^2(|x|\geq R)}^{2}+C_\eta \frac{1}{R^{2(1-s_c)}}\left\{[\rho(u,R)]^{\frac{2-\sigma}{2+\sigma}}+[\rho(u,R)]^{\sigma+1}\right\}\label{GNsmooth}.
	\end{align}
	
Note that in the inequalities \eqref{Young} and \eqref{GNsmooth} we have used
\begin{equation}\label{rema}
\sigma<2\,\,\, \mbox{and}\,\,\,s_c>0, \,\,\, \mbox{since}\,\,\,\frac{2-b}{N}<\sigma<\frac{2-b}{N-2}.
\end{equation}
	
	In the general case, given $u\in \dot{H}^{s_c}\cap\dot{H}^1$ radially symmetric, consider a sequence of smooth radially symmetric functions $u_n$ such that $u_n\to u$ in $\dot{H}^{s_c}\cap \dot{H}^1 $ as $n\to +\infty$. From the Sobolev embedding $\dot{H}^{s_c}\subset L^{\sigma_c}$ it follows that $u_n\to u$ in $L^{\sigma_c}$. Thus, given $\varepsilon>0$ choose $n_{0}\in \mathbb N$ large enough such that for all $n\geq n_0$ we deduce
	\begin{align}\label{conveps}
\|u_n-u\|_{\dot{H}^1}<\varepsilon,\,\,\,\,	\|u_n-u\|_{L^{\sigma_c}}<\varepsilon\,\,\,\,\mbox{ and }\,\,\,\, \int |x|^{-b} |(u_n-u)(x)|^{2\sigma+2}\,dx<\varepsilon,
	\end{align}
where in the last inequality we have used the Gagliardo-Nirenberg inequality \eqref{GNine}. For every $R>0$, since $\sigma_c>2$, by Holder's inequality and \eqref{conveps}, we have
	\begin{align}
	\frac{1}{R^{s_c}}\left|\|u_{n}\|_{L^2(R\leq |x|\leq 2R)}-\|u\|_{L^2(R\leq |x|\leq 2R)}\right|&\leq\frac{1}{R^{s_c}}\|u_{n}-u\|_{L^2(R\leq |x|\leq 2R)}\\
	&\leq \frac{c}{R^{s_c}}\|u_n-u\|_{L^{\sigma_c}}R^{s_c}\\
	& = c\|u_n-u\|_{L^{\sigma_c}}<c\varepsilon,
	\end{align}
where $c>0$ is independent of $R>0$. Therefore
	\begin{align}\label{congrho}
	\rho^{\frac{1}{2}}(u_n,R)\leq c\varepsilon+\rho^{\frac{1}{2}}(u,R).
	\end{align}
Finally, using  \eqref{GNsmooth}, \eqref{conveps} and \eqref{congrho}, we obtain
	\begin{align}
	\int_{|x|\geq R}|x|^{-b}|u|^{2\sigma+2}\,dx&\leq 2^{2\sigma+2}\left(\int_{|x|\geq R}|x|^{-b}|u-u_n|^{2\sigma+2}\,dx+\int_{|x|\geq R}|x|^{-b}|u_n|^{2\sigma+2}\,dx\right)\nonumber\\ &< c\varepsilon+c\eta\|\nabla u_n\|_{L^2(|x|\geq R)}^{2}\\&\quad\quad\quad+\frac{cC_{\eta}}{R^{2(1-s_c)}}\left\{[\rho(u_n,R)]^{\frac{2-\sigma}{2+\sigma}}+[\rho(u_n,R)]^{\sigma+1}\right\}\nonumber\\
	&< c\varepsilon+c\eta\left(\varepsilon+\|\nabla u\|_{L^{2}(|x|\geq R)}\right)^2\\&\quad\quad\quad+\frac{2^{2\sigma+2}C_\eta}{R^{2(1-s_c)}}\left\{[c\varepsilon+\rho^{\frac{1}{2}}(u,R)]^{\frac{4-2\sigma}{2+\sigma}}+[c\varepsilon+\rho^{\frac{1}{2}}(u,R)]^{2\sigma+2}\right\}.\nonumber
	\end{align}
Since $\varepsilon>0$ is arbitrary we concludes the proof of item (ii) and finishes the proof of Lemma \ref{lemaradialGN}.
\end{proof} 

\section{A virial type estimate and the proof of Theorem \ref{thmblowdisper}}\label{sec31}
In this section, we recall a localized version of the viral identity \eqref{VI} and use it to deduce an estimate that plays an important role in the analysis. Indeed, consider a non-negative radial function $\phi\in C^{\infty}_0(\Real^N)$, such that
\begin{align}\label{phi}
\phi(x)=
\left\{
\begin{array}{ll}
\frac{|x|^2}{2},&\mbox{ se }|x|\leq 2\\
0,&\mbox{ se }|x|\geq 4
\end{array}
\right.
\end{align}
satisfying
\begin{align}\label{nablaphi}
\phi(x)\leq c|x|^2, \quad |\nabla \phi (x)|^2\leq c\phi(x) \quad \mbox{and} \quad \partial_r^2\phi(x)\leq 1, \quad \mbox{for all } x\in \Real^N,
\end{align}
with $r=|x|$. Then, define $\phi_R(x)=R^2\phi\left(\frac{x}{R}\right)$. 

Let $v\in C([0,\tau_*],\dot H^{s_c}\cap \dot{H}^1)$ be a solution to \eqref{PVI}. For any $R>0$ and $\tau\in [0,\tau_*]$ define 
\begin{align}
z_R(\tau)=\int \phi_R\left(x\right)|v(x,\tau)|^2\,dx.
\end{align}
Note that, in view of Lemma \ref{lemaradialGN} and Sobolev embedding, $z_R(\tau)$ is well-defined since
\begin{align}
z_R(\tau)=R^2\int_{|x|\leq 4R}\phi\left(\frac{x}{R}\right)|v(\tau)|^2\,dx\leq cR^{2s_c}\|v(\tau)\|_{\dot H^{s_c}}.
\end{align}
Moreover, since $v(\tau)$ is a solution of the INLS equation, we deduce the following virial identities,
\begin{equation}\label{phi'}
z'_R(\tau)=2R \textit{ Im}\int_{\Real^N} \nabla \phi\left(\frac{x}{R}\right)\cdot \nabla v \overline{v}\,dx=2\textit{Im}\int \nabla \phi_R\cdot \nabla v \overline{v}\,dx.
\end{equation}
and
\begin{align}\label{virialseg}
z''_R(\tau)&=4 \sum_{j,k=1}^{N}\int \partial_{x_k} v\,\partial_{x_{j}} \overline{v}\,\partial^2_{x_{j}x_k} \phi\left(\frac{x}{R}\right)\,dx-\frac{1}{R^2}\int |v|^2\Delta^2\phi\left(\frac{x}{R}\right)\,dx\nonumber\\ 
&-\frac{2\sigma}{\sigma+1}\int |x|^{-b}|v|^{2\sigma +2}\Delta\phi\left(\frac{x}{R}\right)\,dx+\frac{2R}{\sigma+1}\int \nabla\left(|x|^{-b}\right)\nabla \phi\left(\frac{x}{R}\right)|v|^{2\sigma+2}\,dx.
\end{align}
(see, for instance, Proposition 7.2 in \cite{FG20}.)

In the next result, we use the identities \eqref{phi'}-\eqref{virialseg} to deduce an important virial type estimate.
\begin{lemma}\label{lemintradial}
	Let $v\in C([0,\tau_*]: \dot H^{s_c}\cap \dot H^1 )$ be a radial solution to \eqref{PVI} with initial data $v_0 \in  \dot H^{s_c}\cap \dot H^1$. Then, there exists $c>0$ depending only on $N, \sigma, b$ such that for all $R>0$ and $\tau\in[0, \tau_\ast]$ we have
	\begin{align}
	2\sigma s_c\int |\nabla v|^2\,dx+\frac{d}{d\tau}\textit{Im}&\int\nabla \phi_R\cdot\nabla v\overline{v}\,dx-4(\sigma s_c+1)E[v_0]\\\leq &\,\,c\left(\frac{1}{R^2}\int_{2R\leq |x|\leq 4R}|v|^2\,dx+\int_{|x|\geq R}|x|^{-b}|v|^{2\sigma+2}\,dx\right).\label{estintradial}
	\end{align}
\end{lemma}

\begin{proof}
	By the virial identity \eqref{virialseg},
	we deduce
	\begin{align}
	\frac{1}{2}\frac{d}{d\tau}\textit{Im}\int \nabla \phi_R\cdot \nabla v \overline{v}\,dx&=\frac{1}{2}\frac{d}{d\tau}\left(\frac{1}{2}\frac{d}{d\tau}z_{R}(\tau)\right)=\frac{1}{4}z''_R(\tau)\\
	&=\sum_{j,k=1}^{N}\int \partial_{x_k} v\,\partial_{x_{j}} \overline{v}\,\partial^2_{x_{j}x_k} \phi\left(\frac{x}{R}\right)\,dx\\
	&\quad-\frac{1}{4R^2}\int |v|^2\Delta^2\phi\left(\frac{x}{R}\right)\,dx\nonumber\\
	&\quad-\frac{\sigma}{2\sigma+2}\int |x|^{-b}|v|^{2\sigma +2}\Delta\phi\left(\frac{x}{R}\right)\,dx\\
	&\quad+\frac{R}{2\sigma+2}\int \nabla\left(|x|^{-b}\right)\nabla \phi\left(\frac{x}{R}\right)|v|^{2\sigma+2}\,dx.\label{virial}
	\end{align}
Recall that for a radially symmetric function $f$ we have 
	\begin{equation}
	\partial_{x_j} f=\frac{x_j}{r}\partial_r f\,\,\,\,\quad\mbox{ and }\quad\,\,\,\,\,\partial^2_{x_jx_k} f=\left(\frac{\delta_{jk}}{r}-\frac{x_jx_k}{r^3}\right)\partial_r f+\frac{x_jx_k}{r^2}\partial^2_r f.
	\end{equation}
	So, since $v$ and $\phi$ are radially symmetric, from the previous relations we obtain
	\begin{align}
	\sum_{j,k=1}^N\int  \partial_{x_k} v\,\partial_{x_{j}} \overline{v}\,\partial^2_{x_{j}x_k} \phi\left(\frac{x}{R}\right)\,dx
	&=\sum_{j,k=1}^N\int  \left|\partial_r v\right|^2\frac{x_jx_k}{r^2}\left(\frac{\delta_{jk}}{r}-\frac{x_jx_k}{r^3}\right)\partial_r \phi \left(\frac{x}{R}\right)\,dx\\
	&\quad+\sum_{j,k=1}^N\int  \left|\partial_r v\right|^2\frac{x_j^2x_k^2}{r^4}\partial^2_r \phi \left(\frac{x}{R}\right)\,dx\\
	&= \int  \left|\partial_r v\right|^2\frac{1}{r}\partial_r \phi\left(\frac{x}{R}\right)\,dx-\int  \left|\partial_r v\right|^2\frac{1}{r}\partial_r \phi\left(\frac{x}{R}\right)\,dx \nonumber\\
	&\quad+\int\left|\partial_r v\right|^2\partial^2_r \phi\left(\frac{x}{R}\right)\,dx\\
	&=\int \partial^2_r \phi\left(\frac{x}{R}\right)|\nabla v|^2\,dx.
	\end{align}
	Thus, using the properties of $\phi$ \eqref{phi}-\eqref{nablaphi}, we have
	\begin{align}
	\frac{1}{4}z''_R(\tau)&=
	\int\partial^2_r \phi \left(\frac{x}{R}\right)|\nabla v|^2\,dx
	-\frac{1}{4R^2}\int_{2R\leq |x|\leq 4R}\Delta^2\phi\left(\frac{x}{R}\right)|v|^2\,dx\nonumber\\
	&\quad-\frac{\sigma}{2\sigma+2}\int |x|^{-b}|v|^{2\sigma+2}\Delta\phi\left(\frac{x}{R}\right)\,dx+\frac{R}{2\sigma+2}\int \nabla\left(|x|^{-b}\right)\cdot\nabla \phi\left(\frac{x}{R}\right)|v|^{2\sigma+2}\,dx\nonumber\\
	&\leq\|\nabla v\|_{L^2}^{2}-\frac{1}{4R^2}\int_{2R\leq |x|\leq 4R}\Delta^2\phi\left(\frac{x}{R}\right)|v|^2\,dx\\
	&\quad-\frac{\sigma}{2\sigma+2}\int_{2R\leq|x|\leq4R}|x|^{-b}|v|^{2\sigma+2}\Delta\phi\left(\frac{x}{R}\right)\,dx\\
	&\quad + \frac{R}{2\sigma+2}\int_{2R\leq|x|\leq 4R}\nabla\left(|x|^{-b}\right)\cdot\nabla\phi\left(\frac{x}{R}\right)|v|^{2\sigma+2}\,dx\nonumber\\
	&\quad-\left(\frac{N\sigma+b}{2\sigma+2}\right)\int_{|x|\leq 2R}|x|^{-b}|v|^{2\sigma+2}\,dx.\nonumber
	\end{align}	
Splitting the last integral in the right hand side in the regions $\{|x|\geq 2R\}$ and $\{|x|\leq 2R\}$, we obtain
\begin{align}
	\frac{1}{4}z''_R(\tau)&= \|\nabla v\|_{L^2}^{2}-\frac{1}{4R^2}\int_{2R\leq |x|\leq 4R}\Delta^2\phi\left(\frac{x}{R}\right)|v|^2\,dx+\left(\frac{N\sigma+b}{2\sigma+2}\right)\int_{|x|\geq 2R}|x|^{-b}|v|^{2\sigma+2}\,dx\nonumber\\
	&\quad-\frac{\sigma}{2\sigma+2}\int_{2R\leq|x|\leq4R}|x|^{-b}|v|^{2\sigma+2}\Delta\phi\left(\frac{x}{R}\right)\,dx\\
	&\quad+\frac{R}{2\sigma+2}\int_{2R\leq|x|\leq4R}\nabla\left(|x|^{-b}\right)\cdot\nabla\phi\left(\frac{x}{R}\right)|v|^{2\sigma+2}\,dx\\
	&\quad-\left(\frac{N\sigma+b}{2\sigma+2}\right)\int |x|^{-b}|v|^{2\sigma+2}\,dx\nonumber\\
	&\leq c\left(\frac{1}{R^2}\int_{2R\leq|x|\leq4R}|v|^2\,dx+\int_{|x|\geq R}|x|^{-b}|v|^{2\sigma+2}\,dx\right)\nonumber\\
	&\quad\quad-\frac{bR}{2\sigma+2}\int_{2R\leq|x|\leq4R}|x|^{-b-2}x\cdot\nabla\phi\left(\frac{x}{R}\right)|v|^{2\sigma+2}\,dx\\
	&\quad\quad+\|\nabla v\|_{L^2}^{2}-\left(\frac{N\sigma+b}{2\sigma+2}\right)\int |x|^{-b}|v|^{2\sigma+2}\,dx.\label{estviri}
	\end{align}
Since $|\nabla \phi(x)|\leq c|x|$ from \eqref{nablaphi}, we deduce
	\begin{align}\label{z''}
	\frac{1}{4}z_R''(t)-\|\nabla v\|_{L^2}^2+&\left(\frac{N\sigma+b}{2\sigma+2}\right)\int |x|^{-b}|v|^{2\sigma+2}\,dx\nonumber\\
	&\leq c\left(\frac{1}{R^2}\int_{2R\leq|x|\leq4R}|v|^2\,dx+\int_{|x|\geq R}|x|^{-b}|v|^{2\sigma+2}\,dx\right).
	\end{align}
	Now, the definition of the energy \eqref{Energy} and $s_c=\frac{N}{2}-\frac{2-b}{2\sigma}$ yield
	\begin{equation}
	\|\nabla v\|_{L^2}^2-\left(\frac{N\sigma+b}{2\sigma+2}\right)\int |x|^{-b}|v|^{2\sigma+2}\,dx= 2(\sigma s_c+1)E[v_0]-\sigma s_c\|\nabla v\|_{L^2}^2.
	\end{equation}
and then  \eqref{z''} implies the desired inequality \eqref{estintradial}.
\end{proof}

Now, we are able to prove Theorem \ref{thmblowdisper}. This is an extension of the result obtained by Merle, Rapha\"{e}l, and Szeftel \cite[Theorem 1.1]{MRS2014} for the classical NLS equation to the INLS model and is a direct consequence of Lemma \ref{lemintradial}. 

\begin{proof}[Proof of Theorem \ref{thmblowdisper}] Let $R,\varepsilon>0$ real number to be chosen later. By the Strauss inequality for radial functions \cite{Strauss} (see, for instance, \cite[Lemma 4.1]{FG20} for a proof) we first have
	\begin{align}
	\|u\|_{L^{\infty}(|x|\geq R)}\leq c\frac{\|\nabla u\|_{L^2}^{\frac{1}{2}}\|u\|_{L^2}^{\frac{1}{2}}}{R^{\frac{N-1}{2}}}.
	\end{align}
	Then, by the mass conservation \eqref{Mass} and Young's inequality
	\begin{align}
	\int_{|x|\geq R}|x|^{-b}|u|^{2\sigma+2}\,dx&\leq \frac{1}{R^{b}}\|u\|_{L^{\infty}(|x|\geq R)}^{2\sigma}\|u\|^{2}_{L^2(|x|\geq R)}\leq \frac{C(u_0)}{R^{\sigma(N-1)+b}}\|\nabla u\|_{L^2}^{\sigma}\|u_0\|_{L^2}^{\sigma+2}\\& \leq \varepsilon \|\nabla u\|_{L^2}^2+\frac{C(u_0,\sigma,\varepsilon)}{R^{\frac{2\sigma(N-1)+2b}{2-\sigma}}},\label{case2}
	\end{align}
where we have used that $\sigma<2$.
	Let $\beta=\frac{2-\sigma}{\sigma(N-1)+b}$. Combining inequality \eqref{estintradial}, energy and mass conservation \eqref{Energy}-\eqref{Mass} and the previous inequality we deduce
	\begin{align}
	2\sigma s_c\int |\nabla u|^2
	\,dx+\frac{d}{d\tau}\textit{Im}&\int \nabla \phi_R\nabla u\overline{u}\,dx\\&\leq C_{u_0}\left(
	1+\frac{1}{R^2}+\int_{|x|\geq R}|x|^{-b}|u|^{2\sigma+2}\,dx\right)\nonumber
	\\
	&\leq C_{u_0}\left(1+\frac{1}{R^2}+\varepsilon\int|\nabla u|^2\,dx+\frac{1}{R^{\frac{2}{\beta}}}\right).
	\end{align}
	Fix $\varepsilon>0$ small enough. If $R\ll1$, since $0<\beta<1$, then
	\begin{align}\label{ctR}
	\sigma s_c\int|\nabla u|^2\,dx+\frac{d}{d\tau}\textit{Im}\int\nabla \phi_R\nabla u\overline{u}\,dx\leq \frac{C_{u_0}}{R^{\frac{2}{\beta}}}.
	\end{align}
	Integrating \eqref{ctR} from $t$ to $\tau $ and then from $t$ to $t_2$, in time, we get
	\begin{align}
	\int_{t}^{t_2}(t_{2}-\tau)\|\nabla u(\tau&)\|_{L^2}^2\,d\tau +\frac{1}{2}\int\phi_R|u(t_2)|^2\,dx\\
	&\leq C_{u_0} \frac{(t_2-t)^2}{2R^{\frac{2}{\beta}}}+\frac{1}{2}\int\phi_R|u(t)|^2\,dx+(t_{2}-t)\textit{Im}\int \nabla \phi_R\nabla u \overline{u}(t)\,dx\nonumber\\
	&\leq C_{u_0}\left(\frac{(t_2-t)^2}{R^{\frac{2}{\beta  }}}+R^2+R(t_2-t)\|\nabla u(t)\|_{L^2} \right),
	\end{align}
where we have used Cauchy-Schwarz inequality and the mass conservation \eqref{Mass} in the last inequality. Taking $t_2\to T^{\ast}$ in the above inequality, we conclude that the first integral on the left hand side converges and
	\begin{align}\label{cotaint}
	\int_{t}^{T^{\ast}}(T^{\ast}-\tau)\|\nabla u(\tau&)\|_{L^2}^2\,d\tau
	\leq C_{u_0}\left(\frac{(T^{\ast}-t)^2}{R^{\frac{2}{\beta}}}+R^2+R(T^{\ast}-t)\|\nabla u(t)\|_{L^2}\right).
	\end{align}
	Thus, choosing $R\ll1$ so that the first two terms on the right hand side are equals, more precisely
	\begin{align}
	R(t)=(T^{\ast}-t)^{\frac{\beta}{1+\beta}},
	\end{align}
the inequality \eqref{cotaint} can be rewritten as
	\begin{align}
	\int_{t}^{T^{\ast}}(T^{\ast}-\tau)\|\nabla u(\tau)\|_{L^2}^2\,d\tau
	&\leq C_{u_0}\left((T^{\ast}-t)^{\frac{2\beta}{1+\beta}}+(T^{\ast}-t)^{\frac{\beta}{1+\beta}}(T^{\ast}-t)\|\nabla u(t)\|_{L^2}\right)\\
	&\leq C_{u_0} (T^{\ast}-t)^{\frac{2\beta}{1+\beta}}+(T^{\ast}-t)^2\|\nabla u(t)\|_{L^2}^2.\label{cotaint2}
	\end{align}
where we have used Young's inequality in the last step.
	
	Define
	\begin{align}
	h(t)=\int_{t}^{T^{\ast}}(T^{\ast}-\tau)\|\nabla u(\tau)\|_{L^2}^{2}\,d\tau.
	\end{align}
	Note that inequality \eqref{cotaint2} implies that $h(0)\leq C_{u_0}$ and moreover
	\begin{align}
	h(t)\leq C_{u_0}(T^{\ast}-t)^{\frac{2\beta}{1+\beta}}-(T^{\ast}-t)h'(t),
	\end{align}
	that is,
	\begin{align}
	\left(\frac{h(t)}{T^{\ast}-t}\right)'=\frac{1}{(T^{\ast}-t)^2}((T^{\ast}-t)h'(t)+h(t))\leq C_{u_0} \frac{1}{(T^{\ast}-t)^{\frac{2}{1+\beta}}}.
	\end{align}
	Integrating the above inequality from $0$ to $t$,
	\begin{align}
	\frac{h(t)}{T^{\ast}-t}\leq \frac{h(0)}{T^{\ast}}+ \frac{C_{u_0}}{(T^{\ast}-t)^{\frac{1-\beta}{1+\beta}}}-\frac{C_{u_0}}{(T^{\ast})^{\frac{1-\beta}{1+\beta}}}\leq \frac{C_{u_0}}{(T^{\ast}-t)^{\frac{1-\beta}{1+\beta}}},
	\end{align}
for $t$ close enough to $T^{\ast}$, since $0<\beta<1 $. Therefore
	\begin{align}
	h(t)\leq C_{u_0}(T^{\ast}-t)^{\frac{2\beta}{1+\beta}},
	\end{align}
which implies \eqref{blowupestdisp} and completes the proof.
\end{proof}

\section{The main propositions}\label{sec4}
In this section, we prove two propositions that are the key pieces in the proof of Theorems \ref{scteo1} and \ref{scteo2}. First, we deduce a control on the semi-norm $\rho$ and the $\dot{H}^1$ norm for radially symmetric solutions to \eqref{PVI}, assuming an initial control on the $L^{\sigma_c}$ norm and the energy.
\begin{prop}\label{prop1}
Let $N\geq 3$, $0<b<\min\left\{\frac{N}{2},2\right\}$, $\frac{2-b}{N}<\sigma<\frac{2-b}{N-2}$ and $v(\tau)\in C\left([0,\tau_\ast]: \dot{H}^{s_c}\cap\dot{H}^1\right)$ a radial solution to \eqref{PVI} with initial data $v_0 \in  \dot H^{s_c}\cap \dot H^1$. Assume
	\begin{equation}\label{hpprop1i}
	\tau_\ast^{1-s_c}\max\{E[v_0],0\}<1
	\end{equation}
	and
	\begin{align}\label{hpprop1ii}
	M_0:=\frac{4\|v_0\|_{L^{\sigma_c}}}{\|V\|_{L^{\sigma_c}}}\geq 2,
	\end{align}
where $V$ is a solution to elliptic equation \eqref{elptcpc1} with minimal $L^{\sigma_c}$-norm. Then, there exist universal constants $C_1,\alpha_1,\alpha_2>0$ depending only on $N,\sigma$ and $b$  such that, for all $\tau_0\in [0,\tau_\ast]$, the following uniform control of the semi-norm $\rho$ holds
	\begin{align}\label{prop1ii}
	\rho(v(\tau_0),M_0^{\alpha_1}\sqrt{\tau_0})\leq C_1M_0^2
	\end{align}
and also the global dispersive estimate
	\begin{align}\label{prop1i}
	\int_0^{\tau_0}(\tau_{0}-\tau)\|\nabla v(\tau)\|_{L^2}^2\,d\tau \leq M_0^{\alpha_2}\tau_{0}^{1+s_c}.
	\end{align}
\end{prop}

Next, under the same hypothesis of Proposition \ref{prop1} and assuming an additional restriction on the energy, we obtain a lower bound on the $L^2$ norm of the initial data in a suitable closed ball around the origin.
\begin{prop}\label{prop2}
Let $N\geq 3$, $0<b<\min\left\{\frac{N}{2},2\right\}$, $\frac{2-b}{N}<\sigma<\frac{2-b}{N-2}$ and $v\in C\left([0,\tau_*]: \dot H^{s_c}\cap \dot H^1\right)$ a radial solution to \eqref{PVI} with initial data $v_0 \in  \dot H^{s_c}\cap \dot H^1$ such that \eqref{hpprop1ii}, \eqref{prop1ii} and \eqref{prop1i} of Proposition \ref{prop1} hold. Let 
	\begin{align}\label{hpprop2i}
	\tau_0\in\left[0,\frac{\tau_*}{2}\right].
	\end{align}
Define $\lambda_v(\tau)=\|\nabla v(\tau)\|^{-\frac{1}{1-s_c}}_{L^2}$ and assume that
	\begin{align}\label{Ev}
	E[v_0]\leq \frac{\|\nabla v(\tau_0)\|^2_{L^2}}{4}=\frac{1}{4\lambda_v^{2(1-s_c)}(\tau_0)}.
	\end{align}
Then, there exist universal  constants $C_2, \alpha_3>0$ depending only on $N,\sigma$ and $b$ such that if
	\begin{align}\label{F}
	F_*=\frac{\sqrt{\tau_0}}{\lambda_v(\tau_0)}\,\,\,\mbox{ and }\,\,\,D_*=M_0^{\alpha_3}\max[1,F_*^{\frac{1+s_c}{1-s_c}}],
	\end{align}
	then
	\begin{align}\label{tprop2}
	\frac{1}{\lambda_v^{2s_c}(\tau_0)}\int_{|x|\leq D_*\lambda_v(\tau_0)}|v_0|^2\,dx\geq C_2.
	\end{align}
\end{prop}
We postpone the proofs of the above propositions to the subsection \ref{5.3.3}. Next, we state and prove some preliminary results. 
\subsection{Auxiliary lemmas}

In the next result we use Lemma \ref{lemaradialGN} and Lemma \ref{lemintradial} to deduce a first control in the quantities that appears in the left hand side of \eqref{prop1i} and in the definition of the semi norm $\rho$ \eqref{defrho} for a suitable choice of $R>0$.


\begin{lemma}\label{lemma34} Let $v\in C([0,\tau_*]: \dot H^{s_c}\cap \dot H^1 )$ be a radially symmetric solution to \eqref{PVI} with initial data $v_0\in \dot H^{s_c}\cap \dot H^1$. For all $A>0$ and $\tau_0\in [0,\tau_{\ast}]$, let $R=A\sqrt{\tau_0}$ and $M_{\infty}$ defined by
	\begin{align}\label{Minfty}
	M^2_{\infty}(A,\tau_0)=\max_{\tau\in[0,\tau_0]}\rho(v(\tau),A\sqrt{\tau}).
	\end{align}
Then, there exists a universal constant $c>0$ such that
	\begin{align}
	2\sigma s_c\int_{0}^{\tau_0}(\tau_0-\tau)&\|\nabla v(\tau)\|_{L^2}^2\,d\tau\\&\leq  c\tau_0^{1+s_c}\left[A^{2(1+s_c)}\|v_0\|_{L^{\sigma_c}}^2+\frac{[M^2_{\infty}(A,\tau_0)]^{\frac{2+\sigma}{2-\sigma}}+M^2_{\infty}(A,\tau_0)}{A^{2(1-s_c)}}\right]\nonumber\\
	&\quad\quad+2\tau_0\left[\textit{Im}\int \nabla\phi_R\cdot\nabla v_0\overline{v_0}\,dx+2\tau_0(\sigma s_c+1) E[v_0]\right] \label{lemi}
	\end{align}
	and
	\begin{align}
	\frac{1}{R^{2s_c}}\int_{R\leq |x|\leq 2R}&|v(\tau_0)|^2\,dx\\&\leq c\|v_0\|_{L^{\sigma_c}}^2+\frac{c}{A^4}\left[[M^2_{\infty}(A,\tau_0)]^{\frac{2+\sigma}{2-\sigma}}+M^2_{\infty}(A,\tau_0)\right]\\
	&\quad\quad+\frac{4}{\tau_0^{s_c}A^{2(1+s_c)}}\left[\textit{Im}\int \nabla\phi_R\cdot\nabla v_0\overline{v_0}\,dx+2\tau_0(\sigma s_c+1)  E[v_0]\right]. \,\,\,\,\,\, \label{lemii}
	\end{align}
\end{lemma}
\begin{proof}
	By the definition of the semi-norm $\rho$ and $M_{\infty}$ with $R=A\sqrt{\tau_0}$, see \eqref{defrho} and \eqref{Minfty} respectively, we have
	\begin{align}
	\rho(v(\tau),R)=\rho(v(\tau),A\sqrt{\tau_0})\leq \rho(v(\tau),A\sqrt{\tau})\leq M_{\infty}^2(A,\tau_0), \,\,\mbox{for all}\,\,\tau \in[0,\tau_0].
	\end{align}
	Thus, we deduce the following estimate
	\begin{align}
	\frac{1}{R^{2}}\int_{2R\leq|x|\leq 4R} |v|^2\,dx&\leq \frac{(2R)^{2s_c}}{R^2}\frac{1}{(2R)^{2s_c}}\int_{2R\leq|x|\leq 4R} |v|^2\,dx\leq \frac{c}{R^{2(1-s_c)}}\rho(v(\tau),2R)\\&\leq \frac{c}{R^{2(1-s_c)}}\rho(v(\tau),R)\leq \frac{c}{R^{2(1-s_c)}}M_{\infty}^2(A,\tau_0).\label{Minfty2}
	\end{align}
	From estimate \eqref{GNradial} with $R=A\sqrt{\tau_0}$, for all $\tau\in [0,\tau_0]$ and $\eta>0$ we obtain
	\begin{align}
	\int_{|x|\geq R}|x|^{-b}|v&|^{2\sigma+2}\,dx\\&\leq \eta\int_{|x|\geq R}|\nabla v|^2\,dx+\frac{C_\eta}{R^{2(1-s_c)}}\left[[\rho(v(\tau),R)]^{\frac{2+\sigma}{2-\sigma}}+[\rho(v(\tau),R)]^{\sigma+1}\right]\nonumber\\
	&\leq \eta\int_{|x|\geq R}|\nabla v|^2\,dx+C_\eta\frac{[M^2_{\infty}(A,\tau_0)]^{\frac{2+\sigma}{2-\sigma}}+[M^2_{\infty}(A,\tau_0)]^{\sigma+1}}{R^{2(1-s_c)}}.\label{GN1}
	\end{align}
	Choosing $\eta>0$ small enough, we inject \eqref{GN1} and \eqref{Minfty2} into \eqref{estintradial} to deduce
	\begin{align}
	\sigma s_c\int|\nabla v |^2\,dx+\frac{d}{d\tau}\textit{Im}\left(\int\nabla \phi_R\cdot\nabla v\overline{v}\,dx\right)
	&-4(\sigma s_c+1)E[v_0]\\&\leq c \frac{[M^2_{\infty}(A,\tau_0)]^{\frac{2+\sigma}{2-\sigma}}+M^2_{\infty}(A,\tau_0)}{R^{2(1-s_c)}},\nonumber
	\end{align}
	where we use $\frac{2+\sigma}{2-\sigma}>\sigma+1>1$.
	Given $\tau\in [0,\tau_0]$ and integrating the above inequality from $0$ to $\tau$, we get
	\begin{align}
	\sigma s_c\int_0^\tau\|\nabla v(s)\|^2_{L^2}\,ds&+\textit{Im}\left(\int \nabla \phi_R\cdot\nabla v(\tau)\overline{v(\tau)}\,dx\right)\\
	&\leq \textit{Im}\left(\int \nabla\phi_R\cdot \nabla v_0\overline{v_0}\,dx\right)+4\tau(\sigma s_c+1)E[v_0]\\
	&\quad +c\tau\frac{[M^2_{\infty}(A,\tau_0)]^{\frac{2+\sigma}{2-\sigma}}+M^2_{\infty}(A,\tau_0)}{R^{2(1-s_c)}}.
	\end{align}
	We integrate one more time from $0$ to $\tau_0$ and use \eqref{phi'} to obtain
	\begin{align}
	\sigma s_c\int_{0}^{\tau_0}\int_0^\tau\|\nabla v(s)\|^2_{L^2}\,ds d\tau&+\frac{1}{2}\int\phi_R|v(\tau_0)|^2\,dx-\frac{1}{2}\int\phi_R|v_0|^2\,dx\\
	&\leq \tau_0\textit{Im}\left(\int \nabla\phi_R\cdot \nabla v_0\overline{v_0}\,dx\right)+2\tau_0^2(\sigma s_c+1)E[v_0]\nonumber\\
	&\quad +c\frac{\tau_0^2}{2}\frac{[M^2_{\infty}(A,\tau_0)]^{\frac{2+\sigma}{2-\sigma}}+M^2_{\infty}(A,\tau_0)}{R^{2(1-s_c)}}.\label{dblinte}
	\end{align}
	Now, note that integration by parts yields
	\begin{align}
	\int_{0}^{\tau_0}\int_{0}^{\tau}\|\nabla v(s)\|_{L^2}^2\,ds d\tau&=\int_{0}^{\tau_0}1\int_{0}^{\tau} \|\nabla v(s)\|_{L^2}^2\,ds d\tau\\
	&=\tau_0\int_{0}^{\tau_0}\|\nabla v(\tau)\|_{L^2}^{2}\,d\tau-\int_{0}^{\tau_0}\tau\|\nabla v(\tau)\|^2_{L^2}\,d\tau\\
	&=\int_{0}^{\tau_0}(\tau_0-\tau)\|\nabla v(\tau)\|_{L^2}^{2}\,d\tau.
	\end{align}
	Then, we can rewrite \eqref{dblinte} as
	\begin{align}
	2\sigma s_c\int_{0}^{\tau_0}(\tau_0-\tau)&\|\nabla v(\tau)\|_{L^2}^{2}\,d\tau+\int \phi_R|v(\tau_0)|^2\,dx\\
	&\leq \int \phi_R|v_0|^2\,dx+2\tau_0\left[\textit{Im}\left(\int \nabla\phi_R\cdot\nabla v_0\overline{v_0}\,dx\right)+2\tau_0(\sigma s_c+1)E[v_0]\right]\nonumber\\
	&\quad+c\tau_0^2\frac{[M^2_{\infty}(A,\tau_0)]^{\frac{2+\sigma}{2-\sigma}}+M_{\infty}^2(A,\tau_0)}{R^{2(1-s_c)}}\\
	& \leq cR^{2(1+s_c)}\|v_0\|_{L^{\sigma_c}}^2+2\tau_0\left[\textit{Im}\left(\int \nabla\phi_R\cdot\nabla v_0\overline{v_0}\,dx\right)+2\tau_0(\sigma s_c+1)E[v_0]\right]\nonumber\\
	&\quad+c\tau_0^2\frac{[M^2_{\infty}(A,\tau_0)]^{\frac{2+\sigma}{2-\sigma}}+M_{\infty}^2(A,\tau_0)}{R^{2(1-s_c)}}\label{prinlem01},
	\end{align}
	where we have used, from inequality \eqref{radial1}, that
	\begin{align}\label{phiR1}
	\int\phi_R|v_0|^2\,dx\leq cR^{2}\int_{|x|\leq 4R}|v_0|^2\,dx\leq cR^{2(1+s_c)}\|v_0\|_{L^{\sigma_c}}^2.
	\end{align}
	Now, since $R=A\sqrt{\tau_0}$, we deduce from \eqref{prinlem01} that
	\begin{align}
	2\sigma s_c\int_{0}^{\tau_0}(\tau_0-\tau)&\|\nabla v(\tau)\|_{L^2}^{2}\,d\tau+\int \phi_R|v(\tau_0)|^2\,dx\\
	& \leq c\tau_0^{1+s_c}A^{2(1+s_c)}\|v_0\|_{L^{\sigma_c}}^2+2\tau_0\left[\textit{Im}\left(\int \nabla\phi_R\cdot\nabla v_0\overline{v_0}\,dx\right)+2\tau_0(\sigma s_c+1)E[v_0]\right]\nonumber\\&\quad+c\tau_0^{1+s_c}\frac{[M^2_{\infty}(A,\tau_0)]^{\frac{2+\sigma}{2-\sigma}}+M_{\infty}^2(A,\tau_0)}{A^{2(1-s_c)}}\label{prinlem},
	\end{align}
	which yields \eqref{lemi}. 
	
	Next, observe that for all $R>0$ the definition of $\phi$ (see \eqref{phi}) yields
	\begin{align}
	\frac{1}{R^{2(1+s_c)}}\int \phi_R|v(\tau_0)|^2\,dx&=\frac{1}{R^{2s_c}}\int_{|x|\leq 4R} \phi\left(\frac{x}{R}\right)|v(\tau_0)|^2\,dx\\
	&\geq \frac{1}{R^{2s_c}}\int_{R\leq |x|\leq 2R}\frac{|x|^2}{2R^2}|v(\tau_0)|^{2}\,dx\geq \frac{1}{2R^{2s_c}}\int_{R\leq|x|\leq 2R}|v(\tau_0)|^2\,dx.
	\end{align}
	Dividing inequality \eqref{prinlem} by $R^{2(1+s_c)}$ where $R=A\sqrt{\tau_0}$, we have
	\begin{align}
	\frac{1}{2R^{2s_c}}\int_{R\leq|x|\leq 2R}&|v(\tau_0)|^2\,dx\\
	&\leq c\|v_0\|_{L^{\sigma_c}}^2+\frac{c}{A^4}\left[[M^2_{\infty}(A,\tau_0)]^{\frac{2+\sigma}{2-\sigma}}+M^2_{\infty}(A,\tau_0)\right]\nonumber\\
	&\quad+\frac{2}{\tau_0^{s_c}A^{2(1+s_c)}}\left[\textit{Im}\left(\int \nabla\phi_R\cdot\nabla v_0\overline{v_0}\,dx\right)+2\tau_0(\sigma s_c+1)E[v_0]\right],
	\end{align}
which implies \eqref{lemii} and concludes the proof.
\end{proof}
In what follows, let $v\in C\left([0,\tau_*]:\dot H^{s_c}\cap \dot H^1 \right)$ be a solution to \eqref{PVI} with initial data $v_0\in \dot H^{s_c}\cap \dot H^1$ and $\varepsilon>0$ a fixed small enough real number to be chosen later. For $M_0$ as in \eqref{hpprop1ii}
, we define
\begin{align}\label{Gep}
G_{\varepsilon}=M_0^{\frac{1}{\varepsilon}}\,\,\, \mbox{and}\,\,\,A_{\varepsilon}=\left(\frac{\varepsilon G_{\varepsilon}}{M_0^2}\right)^{\frac{1}{2(1+s_c)}}.
\end{align}

Consider the following estimates
\begin{align}\label{dispersioneps}
\int_{0}^{\tau_0}(\tau_0-\tau)\|\nabla v(\tau)\|_{L^2}^{2}\,d\tau\leq G_{\varepsilon}\tau_0^{1+s_c}
\end{align}
and 
\begin{align}\label{Minftyep}
M_\infty^2(A_\varepsilon,\tau_0)\leq \frac{2M_0^2}{\varepsilon}.
\end{align}
We define 
\begin{align}\label{tau11}
S_{\varepsilon}= \left\{\tau\in [0,\tau_*];\,\,\eqref{dispersioneps}\mbox{ and }\eqref{Minftyep} \mbox{ hold for all }\tau_0\in [0,\tau]\right\}
\end{align}
and 
\begin{align}\label{tau1}
\tau_1=\max_{\tau\in [0,\tau_*]}S_{\varepsilon}.
\end{align}
Note that $S_{\varepsilon}\neq \emptyset$ for $\varepsilon>0$ sufficiently small. In fact, by the regularity of $v(\tau)$, there exists $0<\delta'<\min\{1,\tau_*\}$ such that
\begin{align}
\|\nabla v(\tau)\|_{L^2}^2<1+\|\nabla v_0\|_{L^2}^2,\,\,\,\mbox{for all}\,\,\, \tau\in [0,\delta'].
\end{align}
Next, multiplying the last inequality by $(\tau_0-\tau)$ and integrating from $0$ to $\tau_0$ with $\tau_0\in [0,\delta']$, we have
\begin{align}
\int_{0}^{\tau_0}(\tau_0-\tau)\|\nabla v(\tau)\|_{L^2}^2\,d\tau\leq \left(\frac{1+\|\nabla v_0\|_{L^2}^2}{2}\right)\tau_0^{1+s_c}\leq M_0^{\frac{1}{\varepsilon}}\tau_0^{1+s_c},
\end{align}
for sufficiently small $\varepsilon>0$, since $M_0\geq 2$. Furthermore, we recall from \eqref{radial1} and \eqref{hpprop1ii} that there exist universal constants $c,C>0$ such that for every $R>0$
\begin{align}
\rho(v_0,R)\leq c\|v_0\|_{L^{\sigma_c}}^2\leq C M_0^2.
\end{align}
Again using the regularity of $v\in C([0,\tau_*]:\dot H^{s_c}\cap\dot H^1 )$ and definition \eqref{Minfty}, there exists $0<\delta''\leq\tau_*$ such that for all $\tau\in [0,\delta'']$
\begin{align}
M_{\infty}^2(A_{\varepsilon},\tau)<2CM_0^2\leq \frac{2M_0^2}{\varepsilon}, 
\end{align} 
for sufficiently small $\varepsilon>0$.
Hence, there exists $\varepsilon_0>0 $, depending only on $v_0$, such that $\min\{\delta',\delta''\}\in S_{\varepsilon}$, for all $\varepsilon\leq \varepsilon_0$. In the next lemma, reducing this $\varepsilon_0>0$ if necessary we obtain an initial control on the virial quantity \eqref{phi'}, assuming an a priori bound on the energy.


\begin{lemma}\label{lemimint}
	Consider $N\geq 3$, $0<b<\min\left\{\frac{N}{2},2\right\}$ and $\frac{2-b}{N}<\sigma<\frac{2-b}{N-2}$. Let $v_0\in \dot H^{s_c}\cap \dot H^1 $ radially symmetric and $v\in C([0,\tau_*]: \dot H^{s_c}\cap \dot H^1 )$ the corresponding solution to \eqref{PVI}. In addition, assume the energy control \eqref{hpprop1i} and let $\tau_1>0$ and $A_\varepsilon>0$ given by \eqref{tau1} and \eqref{Gep}, respectively.
Then, there exists $\varepsilon_1\leq \varepsilon_0$ small enough and $c>0$ a universal constant such that, for all $\tau_0\in[0,\tau_1]$, $0<\varepsilon\leq \varepsilon_1$, $A\geq A_\varepsilon$ and $R=A\sqrt{\tau_0}$, the following inequality holds
	\begin{align}\label{ImE}
	\textit{Im}\int \nabla \phi_R\cdot\nabla v_0\overline{v_0}\,dx+2\tau_0(\sigma s_c+1)E[v_0]\leq c \frac{M_0^2A^{2(1+s_c)}}{\varepsilon^{\frac{1}{1+s_c}}}\tau_0^{s_c}.
	\end{align}
\end{lemma}
\begin{proof}
We split the proof in several steps.\\
	\textbf{Step 1.}
Let $\tau_0\in[0,\tau_1]$ and $\varepsilon\leq \varepsilon_0$. We first claim that there exists a universal constant $K(N,\sigma,b)>0$ and a time $t_0$ such that
	\begin{align}\label{thetaK}
	t_0 \in \left[\frac{\varepsilon^{\frac{1}{1+s_c}}}{4}\tau_0,\frac{\varepsilon^{\frac{1}{1+s_c}}}{2}\tau_0\right]\quad\mbox{ and }\quad \|\nabla v(t_0)\|_{L^2}^2\leq\frac{KG_\varepsilon}{t_0^{1-s_c}}.
	\end{align}
	In fact, let $t=\varepsilon^{\frac{1}{1+s_c}}\tau_0\leq\tau_0\leq\tau_1$. If \eqref{thetaK} does not occur, then from \eqref{dispersioneps} we have
	\begin{align}
	G_{\varepsilon}t^{1+s_c}\geq \int^{t}_{0}(t-\tau)\|\nabla v(\tau)\|_{L^2}^2\,d\tau\geq\frac{t}{2}\int_{\frac{t}{4}}^{\frac{t}{2}}\|\nabla v(\tau)\|_{L^2}^2\,d\tau> K G_{\varepsilon}\frac{t}{2}\int_{\frac{t}{4}}^{\frac{t}{2}}\frac{d\tau}{\tau^{1-s_c}}\geq cKG_{\varepsilon}t^{1+s_c}
	\end{align}
	which yields a contradiction for $K>0$ large enough and thus \eqref{thetaK} is valid.\\
	\textbf{Step 2.} For $A\geq A_{\varepsilon}$, let $A_1>0$ so that $R=A\sqrt{\tau_0}=A_1\sqrt{t_0}$, and thus,
	\begin{align}\label{epsA1}
	\frac{\varepsilon^{\frac{1}{2(1+s_c)}}}{2}\leq \frac{A}{A_1}\leq \varepsilon^{\frac{1}{2(1+s_c)}}.
	\end{align}
	Next, we show the following inequality
	\begin{align}\label{intphi1}
	\frac{1}{R^{2(1+s_c)}}\int \phi_R|v(t_0)|^2\,dx\leq c\left[M_0^2+\frac{G_\varepsilon}{A_1^{2(1+s_c)}}\right].
	\end{align}
	Indeed, by the virial identity \eqref{phi'}, Cauchy-Schwarz inequality and the properties of $\phi$ \eqref{nablaphi} we obtain
	\begin{align}\label{phiRL2}
	\frac{d}{d\tau}\int \phi_R|v|^2\,dx\leq 2\left|\textit{Im}\left(\int\nabla\phi_R\cdot\nabla v\overline v\,dx\right)\right|\leq c\|\nabla v\|_{L^2}\left(\int \phi_R|v|^2\,dx\right)^{\frac{1}{2}}.
	\end{align}
and therefore
	\begin{align}
	\frac{d}{d\tau}\left[\left(\int \phi_R|v(\tau)|^2\,dx\right)^{\frac{1}{2}}\right]\leq c\|\nabla v(\tau)\|_{L^2}.
	\end{align}
	Integrating the above inequality from $0$ to $t_0$, square both sides, using inequality \eqref{phiR1}, the definition of $M_0$ \eqref{hpprop1ii} and Cauchy-Schwarz inequality
	\begin{align}
	\int\phi_R|v(t_0)|^2\,dx&\leq c \left[\int\phi_R|v_0|^2\,dx+\left(\int_0^{t_0}\|\nabla v(\tau)\|_{L^2}\,d\tau\right)^2\right]\\
	&
	\leq c\left[R^{2(1+s_c)}M_0^2+t_0\int_{0}^{t_0}\|\nabla v(\tau)\|_{L^2}^2\,d\tau\right].\label{intphi}
	\end{align}
	Since $2t_0\leq \varepsilon^{\frac{1}{1+s_c}}\tau_0\leq\tau_0\leq\tau_1$ and using estimate \eqref{dispersioneps}, we get
	\begin{align}\label{intzeta0}
	\int_0^{t_0}\|\nabla v(\tau)\|_{L^2}^2\,d\tau\leq \frac{1}{t_0}\int_{0}^{2t_0}(2t_0-\tau)\|\nabla v(\tau)\|_{L^2}^2\,d\tau\leq cG_\varepsilon t_0^{s_c}.
	\end{align}
	Thus, recalling that $R=A_1\sqrt{t_0}$, we combine \eqref{intzeta0} and \eqref{intphi} to obtain
	\begin{align}
	\int\phi_R|v(t_0)|^2\,dx\leq c[R^{2(1+s_c)}M_0^2+G_{\varepsilon}t_0^{1+s_c}]\leq  cR^{2(1+s_c)}\left[M_0^2+\frac{G_\varepsilon}{A_1^{2(1+s_c)}}\right]
	\end{align}
completing the proof of \eqref{intphi1}.\\
	\textbf{Step 3.} Now, we control the virial quantity \eqref{phi'} at time $t_0$. From the second inequality in \eqref{phiRL2}, we apply \eqref{thetaK} and \eqref{intphi1} to deduce for $R=A\sqrt{\tau_0}=A_1\sqrt{t_0}$ that
	\begin{align}
	\left|\textit{Im}\int\nabla\phi_R\cdot\nabla v(t_0)\overline{v(t_0)}\,dx\right|&\leq c\left(\frac{KG_\varepsilon}{t_0^{1-s_c}}\right)^{\frac{1}{2}}R^{1+s_c }\left[M_0^2+\frac{G_\varepsilon}{A_1^{2(1+s_c)}}\right]^{\frac{1}{2}}\\
	&= c\frac{G_\varepsilon^{\frac{1}{2}}}{\left(\tau_0\frac{A^2}{A_1^2}\right)^{\frac{1-s_c}{2}}}A^{1+s_c}\tau_0^{\frac{1+s_c}{2}}\left[M_0+\frac{G_\varepsilon^{\frac{1}{2}}}{A_1^{1+s_c}}\right]\\\nonumber
	&= cM_0^2A^{2(1+s_c)}\tau_0^{s_c}\left[\left(\frac{G_\varepsilon}{A^{2(1+s_c)}M_0^2}\right)^{\frac{1}{2}}\left(\frac{A_1}{A}\right)^{1-s_c}\!\!\!+\frac{G_\varepsilon}{A^{2(1+s_c)}M_0^2}\left(\frac{A}{A_1}\right)^{2s_c}\right].\\
\end{align}
So, for all $A\geq A_\varepsilon$, we deduce
\begin{align}
	\left|\textit{Im}\int\nabla\phi_R\cdot\nabla v(t_0)\right.&\left.\overline{v(t_0)}\,dx\right|\\
&\leq cM_0^2A^{2(1+s_c)}\tau_0^{s_c}\left[\left(\frac{G_\varepsilon}{A_\varepsilon^{2(1+s_c)}M_0^2}\right)^{\frac{1}{2}}\left(\frac{A_1}{A}\right)^{1-s_c}\!\!\!+\frac{G_\varepsilon}{A_\varepsilon^{2(1+s_c)}M_0^2}\left(\frac{A}{A_1}\right)^{2s_c}\right]\\&
	\leq c\frac{M_0^2A^{2(1+s_c)}}{\varepsilon^{\frac{1}{1+s_c}}}\tau_0^{s_c},\label{intImzeta}
	\end{align}
where in the last inequality we have used definition \eqref{Gep} and relation \eqref{epsA1}.\\
	\textbf{Step 4.} Finally, we prove an initial control on the virial quantity \eqref{phi'}. Recall that $A\geq A_{\varepsilon}$, $R=A\sqrt{\tau_0}$ and $2t_0\leq \varepsilon^{\frac{1}{1+s_c}}\tau_0$. From \eqref{Minfty2}, \eqref{Gep} and \eqref{Minftyep}, for all $\tau_0\in [0,\tau_1]$
	\begin{align}
	\left|\int\Delta^2\phi_R|v(\tau_0)|^2\,dx\right|&\leq \frac{c}{R^2}\int_{2R\leq|x|\leq 4R}|v(\tau_0)|^2\,dx\leq c\frac{M_\infty^2(A_\varepsilon,\tau_0)}{R^{2(1-s_c)}}\\&\leq c\frac{M_0^2}{\varepsilon}\frac{\varepsilon^{\frac{1-s_c}{1+s_c}}}{A_\varepsilon^{2(1-s_c)}t_0^{1-s_c}}
	\leq c\frac{M_0^2}{\varepsilon}\frac{M_0^{\frac{2(1-s_c)}{1+s_c}}}{G_\varepsilon^{\frac{1-s_c}{1+s_c}}}\frac{1}{t_0^{1-s_c}}\leq \frac{1}{t_0^{1-s_c}},\label{lapacian2}
	\end{align}
	for $\varepsilon>$ small enough. Choose $\varepsilon_1\leq \varepsilon_0$ such that \eqref{lapacian2} hold. Using the above estimate, the conservation of the energy and the virial estimate \eqref{virial}.
	\begin{align}
	\left|\frac{d}{d\tau}\textit{Im}\left(\int\nabla \phi_R\cdot\nabla v\overline v\,dx\right)\right|\leq c\left[\int|\nabla v|^2\,dx+|E[v_0]|+\frac{1}{t_0^{1-s_c}}\right].
	\end{align}
	Now, integrating from $0$ to $t_0$, we obtain
	\begin{align}
	\left|\textit{Im}\int\nabla \phi_R\cdot\nabla v_0\overline{v_0}\,dx\right|&\leq \left|\textit{Im}\int\nabla \phi_R\cdot\nabla v(t_0)\overline{v(t_0)}\,dx\right|\\&\quad\quad+c\left[\int_0^{t_0}\|\nabla v(\tau)\|_{L^2}^2\,d\tau+t_0|E[v_0]|+t_0^{s_c}\right].
	\end{align}
In view of inequalities \eqref{intzeta0} and \eqref{intImzeta} we deduce
	\begin{align}\label{Imzeta}
	\textit{Im}\int\nabla \phi_R\cdot\nabla v_0\overline{v_0}\,dx \leq& c\frac{M_0^2A^{2(1+s_c)}}{\varepsilon^{\frac{1}{1+s_c}}}\tau_0^{s_c}+ cG_\varepsilon t_0^{s_c}+ct_0|E[v_0]|.
	\end{align}
To estimate the second term in the right hand side of \eqref{Imzeta} we use  \eqref{Gep} and \eqref{thetaK} to obtain
	\begin{align}\label{Gepzeta}
	G_\varepsilon t_0^{s_c}\leq  \frac{A_\varepsilon^{2(1+s_c)} M_0^2}{\varepsilon}\frac{\varepsilon^{\frac{s_c}{1+s_c}}}{2^{s_c}}\tau_0^{s_c}\leq c\frac{M_0^2A^{2(1+s_c)}}{\varepsilon^{\frac{1}{1+s_c}}}\tau_0^{s_c}.
	\end{align}
On the other hand, to control the last term in the right hand side of \eqref{Imzeta}, we observe that
$$
2(\sigma s_c+1)E[v_0]+c\varepsilon^{\frac{1}{1+s_c}}|E[v_0]|\leq c \max\{E[v_0],0\},
$$
for $\varepsilon>0$ small enough. Therefore, using again \eqref{thetaK} and the assumption \eqref{hpprop1i} we have
\begin{align}\label{maxE}
	2\tau_0(\sigma s_c+1)E[v_0]+ct_0|E[v_0]|&\leq \left[2(\sigma s_c+1)E[v_0]+c\varepsilon^{\frac{1}{1+s_c}}|E[v_0]|\right]\tau_0\\\nonumber
	&\leq c\max\{E[v_0],0\}\tau_0^{1-s_c}\tau_0^{s_c}<c\tau_0^{s_c}.
	\end{align}
Finally, collecting \eqref{Imzeta}, \eqref{Gepzeta} and \eqref{maxE} we complete the proof of estimate \eqref{ImE}.
\end{proof}
\subsection{The proofs of Propositions \ref{prop1} and \ref{prop2}}\label{5.3.3}
\begin{proof}[Proof of Proposition \ref{prop1}.] We will show that the estimates \eqref{prop1ii} and \eqref{prop1i} are consequence of \eqref{lemi} and \eqref{lemii} obtained in Lemma \ref{lemma34}. Recall inequalities \eqref{dispersioneps} and \eqref{Minftyep} used in the definition of $\tau_1$ in \eqref{tau1}. We first claim that for all $\tau_0\in [0,\tau_1]$, we have 
	\begin{align}\label{intep1}
	\int_{0}^{\tau_0}(\tau_0-\tau)\|\nabla v(\tau)\|_{L^2}^{2}\,d\tau\leq \frac{G_{\varepsilon}}{2}\tau_0^{1+s_c}
	\end{align}
	and
	\begin{align}\label{intep2}
	M_{\infty}^2(A_{\varepsilon},\tau_0)=\max_{\tau\in [0,\tau_0]}\rho(v(\tau),A_{\varepsilon}\sqrt{\tau})< \frac{M^2_0}{\varepsilon}.
	\end{align}
	for small enough $\varepsilon>0$. Assuming that \eqref{intep1} and \eqref{intep2} hold, by the regularity of $v$, we obtain a contradiction with the definition of $\tau_1$. Therefore inequalities \eqref{dispersioneps} and \eqref{Minftyep} hold at the maximal time $\tau_{\ast}$, which clearly imply \eqref{prop1ii} and \eqref{prop1i}, in view of definitions \eqref{Gep}. 	
%

Thus, it remains to show \eqref{intep1} and \eqref{intep2}. First observe that from \eqref{Gep} and \eqref{Minftyep} we deduce
	\begin{align}\label{110}
	\frac{[M^2_{\infty}(A_{\varepsilon},\tau_0)]^{\frac{2+\sigma}{2-\sigma}}+M^2_{\infty}(A_{\varepsilon},\tau_0)}{A_{\varepsilon}^{2(1-s_c)}}
&\leq \frac{\left(\frac{2M_0^2}{\varepsilon}\right)^{\frac{2+\sigma}{2-\sigma}}+\frac{2M_0^2}{\varepsilon}}{\left(\frac{\varepsilon G_{\varepsilon}}{M_0^2}\right)^{\frac{1-s_c}{1+s_c}}}\leq c\frac{1}{M_0^{\frac{1}{\varepsilon}\cdot \frac{1-s_c}{1+s_c}}} \left(\frac{M_0^2}{\varepsilon}\right)^{\frac{2+\sigma}{2-\sigma}+\frac{1-s_c}{1+s_c}}\leq \frac{1}{10}, 
	\end{align}
for $\varepsilon>0$ small enough, where we have used the assumption \eqref{hpprop1ii}. Combining the last inequality with \eqref{lemi} and using again definitions \eqref{Gep} and assumption \eqref{hpprop1ii}, we have, for $R=A_{\varepsilon}\sqrt{\tau_0}$, that
	\begin{align}
	\int_{0}^{\tau_0}&(\tau_0-\tau)\|\nabla v(\tau)\|_{L^2}^2\,d\tau\\&\leq  c\tau_0^{1+s_c}\left[M_0^2A_{\varepsilon}^{2(1+s_c)}+\frac{[M^2_{\infty}(A_{\varepsilon},\tau_0)]^{\frac{2+\sigma}{2-\sigma}}+M^2_{\infty}(A_{\varepsilon},\tau_0)}{A_{\varepsilon}^{2(1-s_c)}}\right]\nonumber\\
	&+2c\tau_0\left[\textit{Im}\int \nabla\phi_R\cdot\nabla v_0\overline{v_0}\,dx+2\tau_0(\sigma s_c+1) E[v_0]\right]\\
	& \leq c\tau_0^{1+s_c}\left[\varepsilon G_{\varepsilon}+\frac{1}{10}\right]+2c\tau_0\left[\textit{Im}\int \nabla\phi_R\cdot\nabla v_0 \overline{ v_0}\,dx+2\tau_0(\sigma s_c+1) E[v_0]\right]\nonumber\\
	& \leq G_{\varepsilon}\tau_0^{1+s_c}\left\{c\varepsilon+\frac{c}{10G_{\varepsilon}}+\frac{2c}{G_{\varepsilon}\tau_0^{s_c}}\left[\textit{Im}\int \nabla\phi_R\cdot\nabla v_0\overline{v_0}\,dx+2\tau_0(\sigma s_c+1)E[v_0]\right]\right\}\nonumber\\
	&\leq G_{\varepsilon}\tau_0^{1+s_c}\left\{\frac{1}{10}+\frac{2c}{G_{\varepsilon}\tau_0^{s_c}}\left[\textit{Im}\int \nabla\phi_R\cdot\nabla v_0\overline{v_0}\,dx+2\tau_0(\sigma s_c+1)E[v_0]\right]\right\},\nonumber
	\end{align}
for $\varepsilon$ small enough. Now, from Lemma \ref{lemimint} with $A=A_\varepsilon$, for all $\tau_0\in [0,\tau_1]$ and since $R=A_{\varepsilon}\sqrt{\tau_0}$, we get
	\begin{align}
	\int_{0}^{\tau_0}(\tau_0-\tau)\|\nabla v(\tau)\|^2_{L^2}\,d\tau&\leq G_{\varepsilon}\tau_0^{1+s_c}\left[\frac{1}{10}+c\frac{M_0^2A_{\varepsilon}^{2(1+s_c)}}{G_{\varepsilon}\varepsilon^{\frac{1}{1+s_c}}}\right]\\
	&=G_{\varepsilon}\tau_0^{1+s_c}\left[\frac{1}{10}+c\varepsilon^{\frac{s_c}{1+s_c}}\right]\leq \frac{G_{\varepsilon}}{2}\tau_0^{1+s_c},
	\end{align}
and thus \eqref{intep1} is proved. 

To show \eqref{intep2}, let $A\geq A_{\varepsilon}$ and $R=A\sqrt{\tau_0}$. First, since $\rho$ is non-increasing in $R$, then for all $\tau_0\in [0,\tau_1]$ we obtain
	\begin{align}
	M^2_{\infty}(A,\tau_0)\leq M^2_{\infty}(A_{\varepsilon},\tau_0)\leq M^2_{\infty}(A_{\varepsilon},\tau_1)\leq \frac{2M_0^2}{\varepsilon},
	\end{align}
where in the last inequality we have used \eqref{Minftyep}.\\
Combining \eqref{lemii} with \eqref{hpprop1ii}, the last inequality and \eqref{ImE} we deduce
	\begin{align}
	\frac{1}{R^{2s_c}}\int_{R\leq |x|\leq 2R}|v(\tau_0)|^2\,dx&\leq 8cM^2_0+\frac{c}{A^4}\left[[M^2_{\infty}(A,\tau_0)]^{\frac{2+\sigma}{2-\sigma}}+M^2_{\infty}(A,\tau_0)\right]\\
	&\quad +\frac{4}{\tau_0^{s_c}A^{2(1+s_c)}}\left[\textit{Im}\int \nabla\phi_R\cdot\nabla v_0\overline{v_0}\,dx+2\tau_0(\sigma s_c+1) E[v_0]\right]\nonumber\\
	&=8cM_0^2+c\frac{\left(\frac{2M_0^2}{\varepsilon}\right)^{\frac{2+\sigma}{2-\sigma}}+\frac{2M_0^2}{\varepsilon}}{\left(\frac{\varepsilon G_{\varepsilon}}{M_0^2}\right)^{\frac{2}{1+s_c}}}+\frac{4cM_0^2}{\varepsilon^{\frac{1}{1+s_c}}}\\
	&=M_0^2\left[8c+c\frac{1}{M_0^{\frac{1}{\varepsilon}\cdot \frac{2}{1+s_c}}} \left(\frac{M_0^2}{\varepsilon}\right)^{\frac{2+\sigma}{2-\sigma}+\frac{2}{1+s_c}}+\frac{4c}{\varepsilon^{\frac{1}{1+s_c}}}\right]< \frac{M_0^2}{\varepsilon},
	\end{align}
	for $\varepsilon>0$ small enough. The last inequality implies \eqref{intep2} and completes the proof of Proposition \ref{prop1}.
\end{proof}

\begin{proof}[Proof of Proposition \ref{prop2}]
	We first define the following renormalization of $v(\tau_0)$
	\begin{align}
	w(x)=\lambda_v^{\frac{2-b}{2\sigma}}(\tau_0)v(\lambda_v(\tau_0)x,\tau_0),
	\end{align}
where $\lambda_v(\tau_0)=\|\nabla v(\tau_0)\|^{-\frac{1}{1-s_c}}_{L^2}$. Therefore
	\begin{align}\label{wgrad}
	\|\nabla w\|_{L^2}^2=\lambda_v(\tau_0)^{2(1-s_c)}\|\nabla v(\tau_0)\|_{L^2}^2=1.
	\end{align}
Moreover, the conservation of the energy \eqref{Energy} and assumption \eqref{Ev} yield
	\begin{align}
	E[w]&=\frac{1}{2}\int|\nabla w|^2\,dx-\frac{1}{2\sigma+2}\int |x|^{-b}|w|^{2\sigma+2}\,dx\\&=\lambda_v^{2(1-s_c)}(\tau_0)E[v(\tau_0)]=\frac{E[v(\tau_0)]}{\|\nabla v(\tau_0)\|_{L^2}^{2}}\leq \frac{1}{4}.
	\end{align}
	Thus,
	\begin{align}\label{Epg}
	\int |x|^{-b}|w|^{2\sigma+2}\,dx=(2\sigma+2)\left(\frac{1}{2}\int |\nabla w|^2\,dx-E[w]\right)\geq (2\sigma+2)\left(\frac{1}{2}-\frac{1}{4}\right)\geq \frac{\sigma+1}{2}.
	\end{align}
Next recalling the definition of $F_{\ast}$ in \eqref{F}, we set
	\begin{align}\label{Aep}
	J_{*}=C_{*}\max[M_0^{\alpha_1}F_*,M_0^{\frac{2+\sigma}{(1-s_c)(2-\sigma)}}],
	\end{align}
	for $C_*>1$ large enough to be chosen later. 
	
	From the definition of the semi-norm $\rho$ \eqref{defrho}, we get
	\begin{align}
	\rho(w,J_*)&=\sup_{R\geq J_*} \frac{1}{R^{2s_c}}\int_{R\leq |x|\leq 2R}|w|^2\,dx\leq \sup_{R\geq \frac{J_*}{C_*}}\frac{1}{R^{2s_c}}\int_{R\leq |x|\leq2R}|w|^2\,dx\\
	&\leq \sup_{R\geq M_0^{\alpha_1}F_*}\frac{1}{R^{2s_c}}\int_{R\leq |x|\leq 2R} |w|^2\,dx=\sup_{R\geq M_0^{\alpha_1}F_*}\frac{1}{R^{2s_c}}\int_{R\leq |x|\leq 2R}\lambda_v(\tau_0)^{\frac{2-b}{\sigma}}|v(\lambda_v(\tau_0)x,\tau_0)|^2\,dx\nonumber\\
	&=\sup_{R\geq M_0^{\alpha_1}F_*}\frac{1}{(\lambda_v(\tau_0)R)^{2s_c}}\int_{\lambda_v(\tau_0)R\leq |x|\leq 2\lambda_v(\tau_0)R}|v(\tau_0)|^2\,dx\\
	&=\sup_{R\geq \lambda_v(\tau_0)M_0^{\alpha_1}F_*}\frac{1}{R^{2s_c}}\int_{R\leq |x|\leq 2R}|v(x,\tau_0)|^2\,dx\\
	&= \rho(v(\tau_0),M_0^{\alpha_1}\sqrt{\tau_0})\leq c M_0^2,
	\end{align}
where, in the last two steps, we have used the definition of $F_{\ast}$ \eqref{F} and the estimate \eqref{prop1ii} obtained in Proposition \ref{prop1}.

	Thus, for all $\eta>0$ the estimate \eqref{GNradial} implies
	\begin{align}
	\int_{|x|\geq J_*}|x|^{-b}|w|^{2\sigma+2}\,dx&\leq \eta \|\nabla w\|_{L^2}^2+\frac{C_\eta}{J_*^{2(1-s_c)}}\left[[\rho(w,J_*)]^{\frac{2+\sigma}{2-\sigma}}+[\rho(w,J_*)]^{\sigma+1}\right]\\
	&\leq \eta+\frac{C_\eta}{J_*^{2(1-s_c)}}M_0^{\frac{2(2+\sigma)}{2-\sigma}}.
	\end{align}
	Choosing $\eta>0$ small enough and $C_*>0$ large enough (see the definition of $J_{\ast}$ in \eqref{Aep}), from \eqref{Epg} we obtain
	\begin{align}\label{sigma+14}
	\int_{|x|\leq J_*}|x|^{-b}|w|^{2\sigma+2}\,dx\geq\frac{\sigma+1}{4}.
	\end{align}
	
	Considering $\varphi\in C^{\infty}_0(\Real^N)$ a cut-off function such that
	\begin{align}\label{fi}
	\varphi(x)=\left\{
	\begin{array}{ll}
	1,&|x|\leq 1\\
	0,&|x|\geq 2
	\end{array}
	\right.
	\,\,\,\mbox{ and }\,\,\, \varphi_{A}(x)=\varphi\left(\frac{x}{{A}}\right), 
	\end{align}
	we have
	\begin{align}
	\|\nabla (\varphi_{J_*} w)\|_{L^2}&=\|\nabla \varphi_{J_*} w+\varphi_{J_*}\nabla w\|_{L^2}\leq \|\nabla \varphi_{J_*} w\|_{L^2}+\|\varphi_{J_*}\nabla w\|_{L^2}\\&\leq c(\|w\|_{L^2(|x|\leq 2J_*)}+1),
	\end{align}
where we have used \eqref{wgrad} in the last inequality. Applying the Gagliardo-Nirenberg type inequality 
$$
\int|x|^{-b}|f|^{2\sigma+2}\,dx\nonumber\leq c\|\nabla f\|_{L^2}^{2\sigma s_c+2}\|f\|_{L^2}^{2\sigma(1-s_c)}
$$
(see \cite[Theorem 1.2]{Farah}) in the left hand side of \eqref{sigma+14}, we have
	\begin{align}
	\frac{\sigma+1}{4}&\leq \int_{|x|\leq J_*} |x|^{-b}|w|^{2\sigma+2}\,dx\leq \int|x|^{-b}|\varphi_{J_*} w|^{2\sigma+2}\,dx\nonumber\\&\leq c\|\nabla (\varphi_{J_*} w)\|_{L^2}^{2\sigma s_c+2}\|\varphi_{J_*} w\|_{L^2}^{2\sigma(1-s_c)}\leq c\left(\|w\|_{L^2(|x|\leq 2J_*)}^{2\sigma+2}+\|w\|_{L^2(|x|\leq 2J_*)}^{2\sigma(1-s_c)}\right).
	\end{align}
	Therefore, there exists a constant $c_3>0$ such that
	\begin{align}\label{364}
	0<2c_3\leq \int_{|y|\leq 2 J_*}|w|^2\,dx=\frac{1}{\lambda_v^{2s_c}(\tau_0)}\int_{|x|\leq 2J_*\lambda_v(\tau_0)}|v(\tau_0)|^2\,dx.
	\end{align}
	To complete the proof of Proposition \ref{prop2}, we need to show that the above estimate is also verified at $0$. For this purpose, we show that for all $\varepsilon>0$ there exists $\tilde{R}>0$ such that the following inequality holds
	\begin{align}\label{374}
	\frac{1}{\lambda_v^{2s_c}(\tau_0)}\left|\int\varphi_{\tilde{R}}|v(\tau_0)|^2\,dx-\int\varphi_{\tilde{R}}|v_0|^2\,dx\right|<\varepsilon,
	\end{align}
with $\varphi_{\tilde{R}}$ as in \eqref{fi}. More precisely, we show that for all $\varepsilon>0$ there exists $C_{\varepsilon}>0$ such that if
	\begin{align}\label{Dep}
	D_\varepsilon=C_\varepsilon \max\left[F_*,F_*^{\frac{1+s_c}{1-s_c}}\right]\max\left[M_0^{\alpha_1},M_0^{\frac{2+\alpha_2}{2(1-s_c)}}\right],
	\end{align}
and $D\geq D_\varepsilon$, then the inequality \eqref{374} holds for
	\begin{align}\label{Rtil}
	\widetilde{R}=\widetilde{R}(D,\tau_0)=D\lambda_v(\tau_0).
	\end{align}

	Assuming that \eqref{374} is true, from \eqref{364} and since $\|\varphi_R\|_{L^{\infty}}=1$, we have
	\begin{align}
	\frac{1}{\lambda_v^{2s_c}(\tau_0)}\int_{|x|\leq 2\tilde R} |v_0|^2\,dx> \frac{1}{\lambda_v^{2s_c}(\tau_0)}\int\varphi_{\tilde{R}}|v(\tau_0)|^2\,dx-\varepsilon\geq c_3>0,
	\end{align}
for $\varepsilon>0$ small enough. So, choosing $\alpha_3>0$ such that 
$$\frac{M_0^{\alpha_3}}{2}\max[1,F_*^{\frac{1+s_c}{1-s_c}}]\geq D_{\varepsilon},
$$
we deduce the desired inequality (\ref{tprop2}). Thus, it only remains to show \eqref{374}. Indeed, given $\varepsilon>0$, from \eqref{F}, \eqref{Dep} and \eqref{Rtil}, we get for $C_{\varepsilon}>1$
	\begin{align}
\tilde{R}=D\lambda_v(\tau_0)\geq D_\varepsilon\frac{\lambda_v(\tau_0)}{\sqrt{\tau_0}}\sqrt{\tau_0}\geq \frac{D_\varepsilon}{F_*}\sqrt{\tau}\geq M_0^{\alpha_1}\sqrt{\tau}, \,\,\,\mbox{for all}\,\,\, \tau\in [0,\tau_0],
	\end{align}
	and thus, using \eqref{prop1ii} and the monotonicity of $\rho$ we have
	\begin{align}
\rho(v(\tau),\tilde{R})\leq \rho(v(\tau),M_0^{\alpha_1}\sqrt{\tau})\leq cM_0^2, \,\,\,\mbox{for all}\,\,\, \tau\in [0,\tau_0].
	\end{align}
Now, using H\"older's inequality
	\begin{align}
	\left|\frac{d}{d\tau}\int \varphi_{\tilde{R}}|v|^2\,dx\right|=&2\left|\textit{Im}\left(\int \nabla \varphi_{\tilde{R}}\cdot \nabla v\overline{v}\,dx\right)\right|\\
	&\leq \frac{c}{\tilde{R}}\|\nabla v(\tau)\|_{L^2}\left(\int_{\tilde{R}\leq |x|\leq 2\tilde{R}}|v(\tau)|^2\,dx\right)^{\frac{1}{2}}\\
	&\leq c\frac{M_0}{{\tilde{R}}^{1-s_c}}\|\nabla v(\tau)\|_{L^2}.
	\end{align}
	Next, we integrate from $0$ to $\tau_0\in[0,\frac{\tau_*}{2}]$ and use \eqref{prop1i} to obtain
	\begin{align}
	\frac{1}{\lambda_v^{2s_c}(\tau_0)}\left|\int \varphi_{\tilde{R}}|v(\tau_0)|^2\,dx-\int \varphi_{\tilde{R}}|v_0|^2\,dx
	\right|
	&\leq c\frac{D^{2s_c}}{\tilde{R}^{2s_c}}\frac{M_0}{\tilde{R}^{1-s_c}}\int_0^{\tau_0}\|\nabla v(\tau)\|_{L^2}\,d\tau\\
	&\leq c\frac{D^{2s_c}M_0}{\tilde{R}^{1+s_c}}\left(\int_0^{\tau_0}\tau_0\|\nabla v(\tau)\|_{L^2}^2\,d\tau\right)^{\frac{1}{2}}\\
	&\leq c\frac{D^{2s_c}M_0}{\tilde{R}^{1+s_c}}\left(\int_0^{\tau_0}(2\tau_0-\tau)\|\nabla v(\tau)\|_{L^2}^2\,d\tau\right)^{\frac{1}{2}}\\
	&\leq c\frac{D^{2s_c}M_0}{\tilde{R}^{1+s_c}}\left(\int_0^{2\tau_0}(2\tau_0-\tau)\|\nabla v(\tau)\|_{L^2}^2\,d\tau\right)^{\frac{1}{2}}\\
	&\leq cM_0^{\frac{2+\alpha_2}{2}}D^{2s_c}\left(\frac{\tau_0}{\tilde{R}^2}\right)^{\frac{1+s_c}{2}}\\
	& =cM_0^{\frac{2+\alpha_2}{2}}\frac{F_*^{1+s_c}}{D^{1-s_c}},
	\end{align}
where we have used definitions \eqref{F} and \eqref{Rtil} in the last step.

Finally, since $D\geq D_{\varepsilon}\geq C_\varepsilon F_*^{\frac{1+s_c}{1-s_c}}M_0^{\frac{2+\alpha_2}{2(1-s_c)}}$ by \eqref{Dep}, taking $C_\varepsilon>1$ large enough we conclude \eqref{374} and complete the proof of Proposition \ref{prop2}.
\end{proof}
\section{Blow-up of radial solutions with non-negative energy}\label{sec5}
In this section, we prove Theorem \ref{scteo1}, which follows as a direct consequence of Propositions \ref{prop1} and \ref{prop2}.
\begin{proof}[Proof of Theorem \ref{scteo1}] Assume by contradiction that there exists $u\in C([0,+\infty):\dot H^{s_c}\cap \dot H^1 )$ radially symmetric global solution to \eqref{PVI} with non-positive energy. From the Gagliardo-Nirenberg type inequality \eqref{GNine} we deduce that
	\begin{align}
	E[u_0]\geq \frac{1}{2}\|\nabla u_0\|_{L^2}^2\left[1-\left(\frac{\|u_0\|_{L^{\sigma_c}}}{\|V\|_{L^{\sigma_c}}}\right)^{2\sigma}\right].
	\end{align}
	Since $E[u_0]\leq 0$, the conditions \eqref{hpprop1i} and \eqref{hpprop1ii} are satisfied. Thus, for all $\tau_*>0$, we apply Proposition \ref{prop1} and inequality \eqref{prop1i} to deduce
	\begin{align}
	\int_{0}^{\tau_*}(\tau_*-\tau)\|\nabla u(\tau)\|_{L^2}^{2}\,d\tau\leq C_{u_0}\tau_*^{1+s_c}.
	\end{align}
	In particular, we obtain the existence of a sequence $\{\tau_n\}_{n=1}^{+\infty}$ such that $\tau_n\to +\infty$, as $n\to+\infty$, and
	\begin{align}
	\|\nabla u(\tau_n)\|_{L^2}\leq \frac{C_{u_0}}{\tau_n^{\frac{1-s_c}{2}}},\,\,\,\mbox{for all}\,\,\, n\in \mathbb{N}.
	\end{align}
Therefore for $\lambda_u(\tau)=\|\nabla u(\tau)\|^{-\frac{1}{1-s_c}}_{L^2}$ we have
	\begin{align}\label{lambdatau}
	\lambda_u(\tau_n)=\left(\frac{1}{\|\nabla u(\tau_n)\|_{L^2}}\right)^{\frac{1}{1-s_c}}\geq C_{u_0}\sqrt{\tau_n}.
	\end{align}
	Again from the conservation of the energy, we have that the solution $u(\tau)\in \dot H^{s_c}\cap \dot H^1 $ satisfies conditions \eqref{hpprop1i}, \eqref{hpprop1ii} and  \eqref{Ev} with $\tau_0=\tau_n$. Now, applying Proposition \ref{prop2}, inequality \eqref{tprop2}, at time $\tau_n$ and denoting $F_*=F_n = \frac{\sqrt{\tau_n}}{\lambda_u(\tau_n)}$. From \eqref{lambdatau}, $F_n \leq C_{u_0}$ for all $n\in \mathbb{N}$ and hence there exists $D_*$, independent of $n$, such that 
	\begin{align}\label{contrteo1}
	\frac{1}{\lambda_u^{2s_c}(\tau_n)}\int_{|x|\leq D_*\lambda_u(\tau_n)}|u_0|^2\,dx\geq C_2>0,
	\end{align}
in view of inequality \eqref{tprop2}.
	
	On the other hand, in view of \eqref{lambdatau}, $\lambda_u(\tau_n)\to+\infty$, as $n\to+\infty$. Moreover, since $u_0\in \dot H^{s_c} \subset L^{\sigma_c}$, we apply Lemma \ref{lemaradialGN}, more precisely \eqref{radial2} with $R=D_*\lambda_u(\tau_n)$, to obtain 
	\begin{align}
	\lim_{n\to +\infty}\frac{1}{(D_*\lambda_u(\tau_n))^{2s_c}}\int_{|x|\leq D_*\lambda_u(\tau_n)}|u_0|^2\,dx=0,
	\end{align}
which is a contradiction with \eqref{contrteo1}. Therefore, the maximal time of existence $T^{\ast}$ of the solution $u(t)$ is finite and the proof of Theorem \ref{scteo1} is complete.
\end{proof}
\section{A lower bound for the blow-up rate}\label{sec6}
Here, we show that Theorem \ref{scteo2} is also a consequence of Propositions \ref{prop1} and \ref{prop2}. We start with a preliminary and simple result.
\begin{lemma}\label{vteo2}
	Let $N\geq 3$, $0<b<\min\left\{\frac{N}{2},2\right\}$ and $\frac{2-b}{N}<\sigma<\frac{2-b}{N-2}$. 
	Consider $u_0\in\dot H^{s_c}\cap\dot H^{1} $ with radial symmetry and assume that the maximal time of existence $T^{\ast}>0$ of the corresponding solution $u(t)$ to \eqref{PVI} is finite. Also assume that
\begin{align}\label{esttl10}
\lim_{t\to T^{\ast}}\left\|\nabla u(t)\right\|_{L^2}=\infty.
\end{align}
For $t$ close enough to $T^{\ast}$,  consider the following renormalization of $u$
	\begin{align}\label{v}
	v^{(t)}(x,\tau)=\lambda_u^{\frac{2-b}{2\sigma}}(t)\bar{u}(\lambda_u(t)x,t-\lambda^2_u(t)\tau),
	\end{align}
	where $\lambda_u$ is given by
	\begin{align}\label{lambu}
	\lambda_u(t)=\left(\frac{1}{\|\nabla u(t)\|_{L^2}}\right)^{\frac{1}{1-s_c}}.
	\end{align}
	Then, $v^{(t)}(\tau)\in C([0,\tau^{(t)}]:\dot H^{s_c}\cap \dot H^1)$ is a solution to \eqref{PVI} with $\tau^{(t)}=\frac{t}{\lambda_u^2(t)}$ and $v^{(t)}(\tau_0)$ satisfies the conditions of Propositions \ref{prop1} and \ref{prop2} for any $\tau_0$ such that
	\begin{align}
	\tau_0\in\left[0,\frac{1}{\lambda_u(t)}\right].
	\end{align}
\end{lemma}
\begin{proof}
	For $t>0$ fixed, to simplify the notation we write $v=v^{(t)}$. From scaling symmetry $u(x,t)\mapsto \lambda^{\frac{2-b}{2\sigma}}u(\lambda x, \lambda^2 t)$ it is easy to see that $v(\tau)$ is a solution to \eqref{PVI} on the time interval $[0, \frac{t}{\lambda_u^2(t)}]$. Moreover, from the assumption \eqref{esttl10} we get 
	$$
	\frac{t}{\lambda_u^2(t)}\geq \frac{2}{\lambda_u(t)},
	$$
for $t$ close enough to $T^{\ast}$. Set $\tau_\ast=\frac{2}{\lambda_u(t)}$, then $v(\tau)$ is a solution to \eqref{PVI} on the time interval $[0,\tau_\ast]$. Now, we need to check that \eqref{hpprop1i}, \eqref{hpprop1ii}, \eqref{hpprop2i} and \eqref{Ev}  are satisfied. First note that if $\tau_0\in \left[0,\frac{1}{\lambda_u(t)}\right]=\left[0,\frac{\tau_*}{2}\right]$, then \eqref{hpprop2i} holds. Also, it is clear that
	\begin{align}\label{vt}
	\|\nabla v(0)\|_{L^2}=1\,\,\,\,\mbox{ and }\,\,\,\,\,E[v(0)]=\lambda_u(t)^{2(1-s_c)}E[u_0].
	\end{align}
	In particular,
	\begin{align}
	\tau_*^{1-s_c}|E[v(0)]|=\left(\tau_*\lambda_u^2(t)\right)^{1-s_c}|E[u_0]|=(2\lambda_u(t))^{1-s_c}|E[u_0]|\to 0,\,\,\,\mbox{ as }\,\,\,t\to T^{\ast}
	\end{align}
	and \eqref{hpprop1i} follows  for $t$ close enough to $T^{\ast}$. In addition, from the Gagliardo-Nirenberg type inequality \eqref{GNine} and \eqref{vt},
	\begin{align}
	E[v(0)]\geq\frac{1}{2}\left(1-\left(\frac{\|v(0)\|_{L^{\sigma_c}}}{\|V\|_{L^{\sigma_c}}}\right)^{2\sigma}\right),
	\end{align}
which implies \eqref{hpprop1ii}, for $t$ close enough to $T^{\ast}$, since $E[v(0)]\to 0$, as $t\to T^{\ast}$. Next, for all $\tau_0\in \left[0,\frac{1}{\lambda_u(t)}\right]$, we have
	\begin{align}
	\lambda_v(\tau_0)=\left(\frac{1}{\|\nabla v(\tau_0)\|_{L^2}}\right)^{\frac{1}{1-s_c}}=\left(\frac{1}{\lambda_u(t)^{1-s_c}\|\nabla u(t-\lambda_u^2(t)\tau_0)\|_{L^2}}\right)^{\frac{1}{1-s_c}}=\frac{\lambda_u(t-\lambda_u^2(t)\tau_0)}{\lambda_u(t)},
	\end{align}
and
$$
t-\lambda^2_u(t)\tau_0\geq t-\lambda_u(t) \to T^{\ast}, \,\,\,\mbox{ as }\,\,\,t\to T^{\ast}.
$$
Combining these last two relation with \eqref{vt}, we deduce
	\begin{align}
	\lambda_v^{2(1-s_c)}(\tau_0)|E[v(0)]|=&(\lambda_v(\tau_0)\lambda_u(t))^{2(1-s_c)}|E[u_0]|\\
	=&(\lambda_u(t-\lambda_u^2(t)\tau_0))^{2(1-s_c)}|E[u_0]|\to 0, \,\,\,\mbox{ as }\,\,\,t\to T^{\ast},
	\end{align}
and then \eqref{Ev} holds for $t$ close enough to $T^{\ast}$.
\end{proof}

Now, we are ready to complete the proof of Theorem \ref{scteo2}.

\begin{proof}[Proof of Theorem \ref{scteo2}]
We consider $u_0\in \dot H^{s_c}\cap\dot H^1 $ radially symmetric and $u(t)\in \dot H^{s_c}\cap\dot H^1 $ the corresponding solution to \eqref{PVI} with finite maximal time of existence $T^{\ast}>0$. 
Assume that
	\begin{align}\label{lbgradu}
	\|\nabla u(t)\|_{L^2}\geq \frac{c}{(T^{\ast}-t)^{\frac{1-s_c}{2}}}.
	\end{align}

Define
	\begin{align}\label{N}
	N(t)=-\log \lambda_u(t)\,\,\, \mbox{or equivalently}\,\,\, e^{N(t)}=\frac{1}{\lambda_u(t)}.
	\end{align}
and let $v=v^{(t)}$ be as in Lemma \ref{vteo2}. So, we can apply Propositions \ref{prop1} and \ref{prop2} to $v(\tau_0)$ with $\tau_0\in[0,e^{N(t)}]$. Moreover, from the definition of $v$ in \eqref{v}, it is easy to see that
	\begin{align}
	\|u(t)\|_{L^{\sigma_c}}=&\lambda_u^{-\frac{2-b}{2\sigma}}(t)\left(\int |v(\lambda_u(t)^{-1}x,0)|^{\sigma_c}\,dx\right)^{\frac{1}{\sigma_c}}=\lambda_u^{-\frac{2-b}{2\sigma}+\frac{N}{\sigma_c}}\|v(0)\|_{L^{\sigma_c}}=\|v(0)\|_{L^{\sigma_c}}.
	\end{align}
Also, from \eqref{lambu}, \eqref{lbgradu} and \eqref{N} we have
$$
N(t)\geq c|\log (T^{\ast}-t)^{1/2}|=C|\log (T^{\ast}-t)|.
$$
Therefore, to deduce the desired lower bound \eqref{rate}, we reduce our problem to show
\begin{align}\label{N1}
	\|v(0)\|_{L^{\sigma_c}}\geq [N(t)]^{\gamma},
	\end{align}
for some universal constant $\gamma=\gamma(N,\sigma, b)>0$ and for $t$ close enough to $T^{\ast}$. 
	
	To this end, for the constant $\alpha_2$ given in \eqref{prop1i}, define 
	\begin{align}\label{Mt1}
	M(t)=\frac{4\|v(0)\|_{L^{\sigma_c}}}{\|V\|_{L^{\sigma_c}}}\geq 2,\,\,\,\,\,\mbox{ and }\,\,\,\,\,L(t)=[100[M(t)]^{\alpha_2}]^{\frac{1}{2(1-s_c)}}. 
	\end{align}
We consider, without loss of generality, that for all $t$ close enough to $T^{\ast}$
	\begin{align}\label{Lt}
	L(t)<e^{\frac{\sqrt{N(t)}}{2}},
	\end{align}
	otherwise we would have
	\begin{align}
	\|v(0)\|_{L^{\sigma_c}}=\frac{\|V\|_{L^{\sigma_c}}}{4}M(t)=c[L(t)]^{\frac{2(1-s_c)}{\alpha_2}}\geq ce^{\frac{2(1-s_c)}{\alpha_2}\frac{\sqrt{N(t)}}{2}}\geq N(t),
	\end{align}
which implies \eqref{N1}. 

	Now, for all $i\in [\sqrt{N(t)},N(t)]$, we show the existence of times $\tau_i\in[0,e^i]$ such that
	\begin{align}\label{Ftaui}
	F(\tau_i)\leq L(t)\,\,\,\,\mbox{ and }\,\,\,\,\,\frac{1}{10L(t)}e^{\frac{i-1}{2}}\leq \lambda_v(\tau_i)\leq \frac{10}{L(t)}e^{\frac{i}{2}},
	\end{align}
	where $F(\tau)=\sqrt{\tau}\left/\right.\lambda_v(\tau)$ as in \eqref{F}. Let $i\in [\sqrt{N(t)},N(t)]$ and consider
	\begin{align}
	\mathcal{A}=\left\{\tau\in [0,e^{i}];\,\,\,\lambda_v(\tau)>\frac{e^{\frac{i+1}{2}}}{L(t)}\right\}.
	\end{align}
	First, using \eqref{vt} and estimate \eqref{Lt}, we get
	\begin{align}
	\lambda_v(0)=\left(\frac{1}{\|\nabla v(0)\|_{L^2}}\right)^{\frac{1}{1-s_c}}=1<\frac{e^{\frac{\sqrt{N(t)}}{2}}}{L(t)}\leq \frac{e^{\frac{i+1}{2}}}{L(t)}
	\end{align}
	and thus, $0\notin \mathcal{A}$. If $\mathcal{A}$ is nonempty, the continuity of $\lambda_v(\tau)$ implies the existence of $\tau_i\in [0,e^{i}]$ such that $\lambda_v(\tau_i)=\frac{e^{\frac{i+1}{2}}}{L(t)}$. Moreover, 
	\begin{align}
	F(\tau_i)=\frac{\sqrt{\tau_i}}{\lambda_v(\tau_i)}\leq L(t)\frac{\sqrt{e^i}}{e^{\frac{i+1}{2}}}\leq L(t)
	\end{align}
	and
	\begin{align}
	\frac{1}{10}\frac{e^{\frac{i-1}{2}}}{L(t)}\leq e\frac{e^{\frac{i-1}{2}}}{L(t)}=\lambda_v(\tau_i)=e^{\frac{1}{2}}\frac{e^{\frac{i}{2}}}{L(t)}\leq \frac{10e^{\frac{i}{2}}}{L(t)},
	\end{align}
	which proves \eqref{Ftaui}. 
	
On the other hand, if $\mathcal{A}$ is empty, then, for all $\tau\in [0,e^i]$
	\begin{align}
	\lambda_v(\tau)\leq \frac{e^{\frac{i+1}{2}}}{L(t)}\leq \frac{10}{L(t)}e^{\frac{i}{2}}.
	\end{align}
	By contradiction, assume that for all $\tau \in [e^{i-1},e^i]$
	\begin{align}
	F(\tau)\geq L(t),\,\,\,\mbox{ that is }\,\,\,\,
	\|\nabla u(\tau)\|_{L^2}\geq \left(\frac{L(t)}{\sqrt{\tau}}\right)^{1-s_c}.
	\end{align} 
	Thus, from the estimate \eqref{prop1i}, we get
	\begin{align}
	[M(t)^{\alpha_2}]e^{i(1+s_c)}
	&\geq \int_{0}^{e^i}(e^i-\tau)\|\nabla v(\tau)\|^2_{L^2}\, d\tau\\
	&\geq \int_{e^{i-1}}^{\frac{e^i}{2}}\tau\|\nabla v(\tau)\|^2_{L^2}\,d\tau\\
	&\geq [L(t)]^{2(1-s_c)}\int_{e^{i-1}}^{\frac{e^i}{2}}\frac{\tau}{\tau^{1-s_c}}\,d\tau\\
	&=[L(t)]^{2(1-s_c)}\frac{e^{i(1+s_c)}}{1+s_c}\left(\frac{1}{2^{1+s_c}}-\frac{1}{e^{1+s_c}}\right)\\
	&>\frac{[L(t)]^{2(1-s_c)}}{100}e^{i(1+s_c)}\\
	&=[M(t)]^{\alpha_2}e^{i(1+s_c)},
	\end{align}
where we have used in the last step the definition of $L(t)$ in \eqref{Mt1}. Thus, we obtain a contradiction. Consequently, there exists $\tau_i\in[e^{i-1},e^i]$ such that 
	$
	F(\tau_i)\leq L(t).
	$
Moreover
	\begin{align}
	\lambda_v(\tau_i)=\frac{\sqrt{\tau_i}}{F(\tau_i)}\geq \frac{\sqrt{e^{i-1}}}{L(t)}\geq \frac{e^{\frac{i-1}{2}}}{10L(t)}
	\end{align}
	and \eqref{Ftaui} is proved.

Next, we prove an uniform lower bound for the $L^{\sigma_c}$ norm over a suitable annulus on space. Let $i\in [\sqrt{N(t)},N(t)]$ be an integer and $\tau_i\in[0,e^i]$ the times satisfying \eqref{Ftaui}. We show the existence of universal positive constants $\alpha_4(N,\sigma, b), \, c_4(N,\sigma,b)>0$ such that in the annulus
	\begin{align}\label{Rti1}
	\mathcal{C}_i=\left\{x\in \Real^N;\, \frac{\lambda_v(\tau_i)}{[M(t)]^{\alpha_4}}\leq |x|\leq [M(t)]^{\alpha_4}\lambda_v(\tau_i)\right\},
	\end{align}
	we have
	\begin{align}\label{teste}
	\int_{\mathcal{C}_i}|v(0)|^{\sigma_c}\,dx\geq \frac{c_4}{[M(t)]^{\alpha_4s_c\sigma_c}}.
	\end{align}
Indeed, from \eqref{Mt1} and \eqref{Ftaui} we have
	\begin{align}
	F(\tau_i)\leq L(t)=[100[M(t)]^{\alpha_{2}}]^{\frac{1}{2(1-s_c)}}.
	\end{align}
Moreover
\begin{equation}\label{Dast}
D_*(t):=[M(t)]^{\alpha_3}[F(\tau_i)]^{\frac{1+s_c}{1-s_c}}\leq [M(t)]^{\alpha_4},
\end{equation}
for some $\alpha_4(N,\sigma_c)>0$, independent of $t$. Thus, the choice of $D_*(t)$ in \eqref{F} is uniform with respect to $i$ and, in view of Lemma \ref{vteo2}, we can apply Proposition \ref{prop2} to $v(\tau_i)$. Therefore, for all $i\in[\sqrt{N(t)},N(t)]$ we apply H\"older's inequality, \eqref{Dast}, \eqref{tprop2} and \eqref{Mt1}, to get
	\begin{align}
	\int_{\frac{\lambda_v(\tau_i)}{[M(t)]^{\alpha_4}}\leq |x|\leq [M(t)]^{\alpha_4}\lambda_v(\tau_i)}|v(0)|^2\,dx
	=&\int_{|x|\leq [M(t)]^{\alpha_4}\lambda_v(\tau_i)}|v(0)|^2\,dx-
	\int_{|x|\leq\frac{\lambda_v(\tau_i)}{[M(t)]^{\alpha_4}}}|v(0)|^2\,dx\\
	\geq&\int_{|x|\leq D_*(t)\lambda_v(\tau_i)}|v(0)|^2\,dx-
	\frac{c\lambda_v^{2s_c}(\tau_i)}{[M(t)]^{2s_c\alpha_4}}\left\|v(0)\right\|^{2}_{L^{\sigma_c}}\\
	\geq& \lambda_v^{2s_c}(\tau_i)\left(C_2-\frac{cM^2(t)}{[M(t)]^{2s_c\alpha_4}}\right)
	\geq\lambda^{2s_c}_v(\tau_i)\frac{C_2}{2},
	\end{align}
for $\alpha_4>0$ large enough\footnote{Note again that this choice is independent of $t$ since $M(t)\geq 2$ for all $t\in [0,T^{\ast})$.}. Moreover, again H\"older's inequality yields
	\begin{align}
	\int_{\frac{\lambda_v(\tau_i)}{[M(t)]^{\alpha_4}}\leq |x|\leq [M(t)]^{\alpha_4}\lambda_v(\tau_i)}|v(0)|^2\,dx\leq c\lambda^{2s_c}_v(\tau_i)[M(t)]^{2s_c\alpha_4}\left(\int_{\mathcal{C}_i}|v(0)|^{\sigma_c}\,dx\right)^{\frac{2}{\sigma_c}}.
	\end{align}
Combining the last two inequalities we conclude the proof of \eqref{teste}.

Now, let  $p(t)>1$ an integer such that 
	\begin{align}\label{cotapt}
	10^3[M(t)]^{2\alpha_4}\leq e^{\frac{p(t)}{2}}\leq 10^7[M(t)]^{2\alpha_4}.
	\end{align}
	Note that, if $p(t)>\sqrt{N(t)}$, then by definition \eqref{Mt1}
	\begin{align}
	\|v(0)\|_{L^{\sigma_c}}^{2\alpha_4}=c[M(t)]^{2\alpha_4}\geq ce^{\frac{p(t)}{2}}\geq p(t)\geq \sqrt{N(t)},
	\end{align}
which implies \eqref{N1}. On the other hand, consider $p(t)<\sqrt{N(t)}$. Let $(i,i+p(t))\in [\sqrt{N(t)},N(t)]\times [\sqrt{N(t)},N(t)]$, then by \eqref{Ftaui} and \eqref{cotapt}
	\begin{align}
\lambda_v(\tau_{i+p(t)})&\geq \frac{e^{\frac{i+p(t)-1}{2}}}{10L(t)}
	\geq\frac{10^3[M(t)]^{2\alpha_4}e^{\frac{i-1}{2}}}{10L(t)}\geq [M(t)]^{2\alpha_4}\frac{10e^{\frac{i}{2}}}{L(t)}\geq [M(t)]^{2\alpha_4}\lambda_v(\tau_i),
	\end{align}
and, dividing the previous inequality by $[M(t)]^{\alpha_4}$, we deduce that the annuli $\mathcal{C}_i$ and $\mathcal{C}_{i+p(t)}$ given by \eqref{Rti1} are disjoints for each $(i,i+p(t))\in [\sqrt{N(t)},N(t)]\times [\sqrt{N(t)},N(t)]$.
	
	Since $p(t)<\sqrt{N(t)}$ and $\frac{1}{10}<1-\frac{\sqrt{N(t)}}{N(t)}$ for $t$ close enough to $T^{\ast}$, there are at least $\frac{N(t)}{10p(t)}\geq \frac{1}{10}\sqrt{N(t)}$ disjoint annuli satisfying the uniform lower bound \eqref{teste}. Hence, summing over these sets, we finally deduce
	\begin{align}
	\|v(0)\|_{L^{\sigma_c}}^{\sigma_c}\geq \sum_{k=0}^{\frac{N(t)}{10p(t)}}\int_{\mathcal{C}_{kp(t)+\sqrt{N(t)}}}|v(0)|^{\sigma_c}\,dx\geq \frac{c_4}{[M(t)]^{\alpha_4\sigma_cs_c}}\frac{\sqrt{N(t)}}{10}=C\frac{\sqrt{N(t)}}{\|v(0)\|_{L^{\sigma_c}}^{\alpha_4\sigma_cs_c}},
	\end{align}
where we have used \eqref{Mt1} in the last step. The previous inequality implies \eqref{N1} and completes the proof of Theorem \ref{scteo2}.
\end{proof}

Finally, we prove our last result, which asserts that even without the assumption \eqref{esttl} the $\dot H^{s_c}$ norm of finite time radially symmetric solutions in $\dot H^{s_c}\cap \dot H^1$ is never uniformly bounded in time.

\begin{proof}[Proof of Corollary \ref{cor13}]
Let $u\in C([0,T^{\ast}): \dot H^{s_c}\cap \dot H^1 )$ be a  solution to \eqref{PVI} with finite maximal time existence $T^{\ast}>0$ such that
	\begin{align}\label{Msup}
	\sup_{t\in [0,T^{\ast})}\|u(t)\|_{\dot H^{s_c}}=M<+\infty,
	\end{align} 
Then for any $t\in [0,T^{\ast})$, we consider the following scaling of $u$
	\begin{align}\label{scaling}
	v^{(t)}(x,\tau)=\rho^{\frac{2-b}{2\sigma}}(t)u(\rho(t)x,t+\rho^2(t)\tau),
	\end{align}
	where $\rho(t)^{1-s_c}\|\nabla u(t)\|_{L^2}=1$. Hence,
$$
\|v^{(t)}(0)\|_{\dot H^{s_c}}=\|u(t)\|_{\dot H^{s_c}}\leq M\,\,\,\mbox{ and }\,\,\,\,\|\nabla v^{(t)}(0)\|_{L^2}=\rho(t)^{1-s_c}\|\nabla u(t)\|_{L^2}=1
$$
and thus, from the local Cauchy theory in $\dot H^{s_c}\cap\dot H^1 $ (see \cite[Theorem 1.2.]{CFG20}), there exists $\tau_0$, independent of $t,$ such that $v^{(t)}$ is defined on $[0,\tau_0]$. Therefore, $t+\rho^2(t)\tau_0<T^{\ast}$ and this is \eqref{esttl}. From Theorem \ref{scteo2}, the lower bound \eqref{rate} holds and, in particular, we have that
$$
\lim_{t\uparrow T^{\ast}}\|u(t)\|_{\dot H^{s_c}}=+\infty,
$$
which is a contradiction with \eqref{Msup}. 
\end{proof}

\vspace{0.5cm}
\noindent 
\textbf{Acknowledgments.} M.C. was partially supported by Coordena\c{c}\~ao de Aperfei\c{c}oamento de Pessoal de N\'ivel Superior - CAPES. L.G.F. was partially supported by Coordena\c{c}\~ao de Aperfei\c{c}oamento de Pessoal de N\'ivel Superior - CAPES, Conselho Nacional de Desenvolvimento Cient\'ifico e Tecnol\'ogico - CNPq and Funda\c{c}\~ao de Amparo a Pesquisa do Estado de Minas Gerais - Fapemig/Brazil. 



\end{document}